\documentclass[12pt]{amsart}
\usepackage{amssymb,amsmath,amsfonts,epsfig,latexsym,xypic,hyperref,comment}

\usepackage{mathbbol}

\DeclareSymbolFontAlphabet{\mathbbm}{bbold}
\DeclareSymbolFontAlphabet{\mathbb}{AMSb}%

\usepackage{enumerate}
\usepackage{tikz}
\usepackage{mathrsfs}

\tikzset{my_dot/.style={fill, circle, inner sep=0pt,minimum size=3pt}}
\tikzset{my_node/.style={fill, circle, inner sep=0pt,minimum size=3pt}}
\tikzset{inv/.style={fill, circle, inner sep=0pt,minimum size=0pt}}

\newtheorem{theorem}{Theorem}[section]
\newtheorem{definition}[theorem]{Definition}
\newtheorem{proposition}[theorem]{Proposition}
\newtheorem{lemma}[theorem]{Lemma}
\newtheorem{corollary}[theorem]{Corollary}

\numberwithin{equation}{subsection}

\theoremstyle{definition}
\newtheorem{remark}[theorem]{Remark}
\newtheorem{example}[theorem]{Example}

\def\C{\ensuremath{\mathbb{C}}}

\def\Q{\ensuremath{\mathbb{Q}}}
\def\R{\ensuremath{\mathbb{R}}}
\def\Z{\ensuremath{\mathbb{Z}}}

\def\cD{\ensuremath{\mathcal{D}}}
\def\cM{\ensuremath{\mathcal{M}}}

\def\cX{\ensuremath{\mathcal{X}}}

\DeclareMathOperator{\Aut}{Aut}

\DeclareMathOperator{\colim}{colim}

\DeclareMathOperator{\Gr}{Gr}
\DeclareMathOperator{\Hom}{Hom}
\DeclareMathOperator{\Mod}{Mod}
\DeclareMathOperator{\Out}{Out}

\DeclareMathOperator{\val}{val}

\def\col{\colon}
\def\Dg{\Delta_g}

\def\Gmw{{\bf G}}
\def\inj{\mathrm{inj}}

\def\Jg{\mathbbm{\Gamma}_{g}}
\def\JX{J_{X}}

\def\Mg{M^{\trop}_{g}} 
\def\ocM{\overline{\cM}}
\def\ocX{\overline{\cX}}
\def\ov{\overline}

\def\trop{\mathrm{trop}}

\def\Sets{\mathrm{Sets}} 

\newcommand{\double}{\genfrac..{0pt}1
{\raise -2pt\hbox{$\scriptstyle\longrightarrow$}}{\raise 4pt\hbox
{$\scriptstyle\longrightarrow$}}}

\begin{document}

\title{Tropical curves, graph complexes, and top weight cohomology of $\cM_g$}

\author{Melody Chan}
\email{melody\_chan@brown.edu}

\author{S{\o}ren Galatius}
\email{galatius@math.ku.dk}

\author{Sam Payne}
\email{sampayne@utexas.edu}

\begin{abstract}
  We study the topology of a space $\Dg$ parametrizing stable tropical curves of genus $g$ with volume $1$, showing that its reduced rational homology is canonically identified with both the top weight cohomology of $\cM_g$ and also with the genus $g$ part of the homology of Kontsevich's graph complex.  Using a theorem of Willwacher relating this graph complex to the Grothendieck--Teichm\"uller Lie algebra, we deduce that $H^{4g-6}(\cM_g;\Q)$ is nonzero for $g=3$, $g=5$, and $g \geq 7$, and in fact its dimension grows at least exponentially in $g$.  This disproves a recent conjecture of Church, Farb, and Putman as well as an older, more general conjecture of Kontsevich.  We also give an independent proof of another theorem of Willwacher, that homology of the graph complex vanishes in negative degrees.
 \end{abstract}

\maketitle

\tableofcontents

\section{Introduction}

Fix an integer $g\ge2$. In this paper, we study the topology of a space $\Dg$ that parametrizes isomorphism classes of genus $g$ {\em tropical curves} of volume 1.  Tropical curves are certain weighted, marked metric graphs; see \S\ref{subsec:J-g-n} for the precise definition.
  
Interest in the space $\Dg$ is not limited to tropical geometry. Indeed, $\Dg$ may be identified homeomorphically with the following spaces: 

\begin{enumerate}
\item the link of the vertex in the tropical moduli space $\Mg$ \cite{acp, BrannettiMeloViviani11};

\item the dual complex of the boundary divisor in $\ocM_{g}$, the algebraic moduli space of stable curves of genus $g$ (Corollary~\ref{cor_dual_complex});

\item the quotient of the simplicial completion of Culler--Vogtmann outer space by the action of the outer automorphism group $\Out(F_g)$ \cite[\S 5.2]{ConantVogtmann03}, \cite[\S 2.2]{Vogtmann15}; 

\item the topological quotient of Harvey's complex of curves on a surface of genus $g$ by the action of the mapping class group \cite{Harvey81}; and

\item the topological quotient of Hatcher's complex of sphere systems in certain a 3-manifold  by the action of the mapping class group of that manifold \cite{Hatcher95}.
\end{enumerate}

\noindent Our primary focus will be on the interpretations (1) and especially (2) from tropical and algebraic geometry: we apply combinatorial topological calculations on $\Delta_g$ to compute previously unknown invariants of the complex algebraic moduli space $\cM_g$.  One such application gives a lower bound on the size of $H^{4g-6}(\cM_g,\Q)$, as follows.

\begin{theorem}  \label{thm:nonvanishing}
The cohomology $H^{4g-6}(\cM_g;\Q)$ is nonzero for $g =3$, $g = 5$, and $g \geq 7$.  
Moreover, $\dim H^{4g-6}(\cM_g;\Q)$ grows at least exponentially.  More precisely, $$\dim H^{4g-6}(\cM_g;\Q) > \beta^g + \text{constant}$$ for any $\beta< \beta_0$, where $\beta_0 \approx 1.3247\ldots$ is the real root of $t^3-t-1=0$. 
\end{theorem}

\noindent The nonvanishing for $g = 3$ was known previously; Looijenga famously showed that the unstable part of $H^6(\cM_3;\Q)$ has rank 1 and weight 12 \cite{Looijenga93}.  

To put Theorem~\ref{thm:nonvanishing} in context, recall that the virtual cohomological dimension of $\cM_g$ is $4g-5$ \cite{Harer86}.  Church, Farb, and Putman conjectured that, for each fixed $k > 0$, $H^{4g-4-k}(\cM_g;\Q)$ vanishes for all but finitely many $g$ \cite[Conjecture~9]{ChurchFarbPutman14}.  While this is true for $k = 1$  \cite{ChurchFarbPutman12, MoritaSakasaiSuzuki13},   Theorem~\ref{thm:nonvanishing} shows that it is false for $k = 2$.  
Furthermore, as observed by Morita, Sakasai, and Suzuki \cite[Remark~7.5]{MoritaSakasaiSuzuki15}, the Church-Farb-Putman conjecture is implied by a more general statement conjectured by Kontsevich two decades earlier \cite[Conjecture~7C]{Kontsevich93}, which we now recall.  In the same paper where he introduced the graph complex, Kontsevich studied three infinite dimensional Lie algebras, whose homologies are free 
graded commutative algebras generated by subspaces of primitive elements.  Each contains the primitive homology of the Lie algebra $\mathrm{sp}(2\infty)$ as a direct summand.  
For one of these Lie algebras, denoted $a_\infty$, the complementary primitive homology is
\[
PH_k(a_\infty) / PH_k (\mathrm{sp}(2\infty)) \ \cong  \bigoplus_{m >0, 2g-2+m >0} H^{4g-4+2m-k} (\mathcal{M}_{g,m}/S_m; \Q),
\]
where $S_m$ denotes the symmetric group acting on the moduli space $\cM_{g,m}$ of curves with $m$ marked points by permuting the markings.  See \cite[Theorem~1.1(2)]{Kontsevich93}. 

Kontsevich conjectured that the homology of each of these Lie algebras should be finite dimensional in each degree.  In particular, for each $k$, the cohomology group $H^{4g-2-k}(\cM_{g,1};\Q)$ should vanish for all but finitely many $g$.  
Note that the composition
$$H^*(\cM_g;\Q)\rightarrow H^*(\cM_{g,1};\Q)\rightarrow H^{*+2}(\cM_{g,1};\Q),$$
where the second map is cup product with the Euler class, is injective. This is because further composing with Gysin pushforward to $H^*(\cM_g;\Q)$ gives multiplication by $2-2g < 0$.  Therefore, 
Theorem~\ref{thm:nonvanishing} shows that $PH_2(a_\infty)$ is infinite dimensional, disproving Kontsevich's conjecture and giving a negative answer to \cite[Problem~7.4]{MoritaSakasaiSuzuki15}.

Theorem~\ref{thm:nonvanishing}, and further applications discussed in Section~\ref{sec:applications}, will be established via combinatorial topological calculations on the space $\Delta_g$, which may be identified with the {\em dual complex} of the Deligne-Mumford stable curve compactification $\ocM_g$ of $\cM_g$.  
Throughout, we work with varieties and Deligne-Mumford stacks over $\C$.
Recall that Deligne has defined a natural weight filtration on the rational singular cohomology of any complex algebraic variety which gives, together with the Hodge filtration on singular cohomology with complex coefficients, a mixed Hodge structure \cite{Deligne71, Deligne74b}.
The graded pieces of the weight filtration on the cohomology of a $d$-dimensional variety are supported in degrees between $0$ and $2d$, and we refer to the $2d$-graded piece, denoted $\Gr_{2d}^W$, as the {\em top weight cohomology}.  We will use the interpretation of $\Delta_{g}$ as the dual complex of the boundary divisor in the Deligne--Mumford compactification of $\cM_{g}$ to give an identification of its reduced rational homology with the top weight-graded piece of the cohomology of $\cM_{g}$.  

\begin{theorem}  \label{thm:comparison}
There is an isomorphism
\[
\Gr_{6g-6}^W H^{6g-6-k} (\cM_{g}; \Q) \xrightarrow{\cong} \widetilde{H}_{k-1}(\Dg;\Q) ,
\]
identifying the top graded piece of the weight filtration on the cohomology of $\cM_{g}$ with the reduced rational homology of $\Dg$.
\end{theorem}

\newcommand{\ellmap}{\lambda}

\noindent Our proof of Theorem \ref{thm:comparison} produces a specific isomorphism, which is in fact induced by a proper map of topological spaces
\begin{equation}\label{eq:4}
  \cM_g \stackrel{\ellmap}\longrightarrow M_g^\trop,
\end{equation}
defined using the hyperbolic model for $\cM_g$, see \S\ref{sec:induced-by-map-of-spaces}.  Compactly supported cohomology is functorial with respect to proper maps, so~\eqref{eq:4} induces maps $\ellmap^* \colon H^k_c(M_g^\trop;\Q) \to H^k_c(\cM_g;\Q)$.  Using Poincar\'e duality for $\cM_g$ and that $\Delta_g$ is the link of the cone point in $M_g^\trop$, this gives
\begin{equation*}
  H^{6g-6-k}(\cM_g;\Q) \cong (H^k_c(\cM_g;\Q))^\vee \stackrel{\ellmap_*}{\longrightarrow}
  (H^k_c(M_g^\trop;\Q))^\vee \cong \widetilde{H}_{k-1}(\Delta_g;\Q).
\end{equation*}
We will see that this is a surjection, factoring over the isomorphism stated in Theorem~\ref{thm:comparison}.  In this sense~\eqref{eq:4} will be a space-level refinement of the map in rational cohomology.

\medskip
The space $\Dg$ is glued out of standard simplices $\Delta^p$ in a way that resembles $\Delta$-complexes, except simplices may be glued to themselves along permutations of their vertices.  We call such objects \emph{symmetric $\Delta$-complexes}, briefly reviewed in Section~\ref{sec:cellular}.  In particular, its rational homology may be calculated by a cellular chain complex.
We will relate the cellular chain complex computing reduced rational homology of $\Delta_g$ to the \emph{commutative graph complex} $G^{(g)}$, introduced by Kontsevich \cite{Kontsevich93, Kontsevich94}.  The precise definition of $G^{(g)}$ is recalled in Section~\ref{sec:konts-graph-compl} and shall not be needed here in the introduction.  The graph complex $G^{(g)}$ has been studied intensively, including in the past few years. See, e.g.,
\cite{ConantVogtmann03, ConantGerlitsVogtmann05, DolgushevRogersWillwacher15, Willwacher15}.

We will construct a quasi-isomorphism from the cellular chain complex computing $\widetilde{H}_*(\Delta_g)$ to $G^{(g)}$.
Passing to homology gives the following:
\begin{theorem}  \label{thm:gc}
  For $g \geq 2$, there is an isomorphism $${H}_k(G^{(g)}) \xrightarrow{\cong} \widetilde{H}_{2g+k-1}(\Delta_{g};\Q).$$
\end{theorem}

Combining Theorems~\ref{thm:comparison} and \ref{thm:gc} then gives a surjection $H^{4g-6-k}(\cM_{g};\Q) \twoheadrightarrow H_{k}(G^{(g)}).$
In particular, nonvanishing graph homology groups yield nonvanishing results for cohomology of $\cM_g$.
The full structure of the homology of the graph complex remains mysterious, but several interesting substructures and many nontrivial classes are known and understood. In particular, the linear dual of $\bigoplus_g H_0(G^{(g)})$ carries a natural Lie bracket, and is isomorphic to the Grothendieck-Teichm\"uller Lie algebra $\mathfrak{grt}_1$ by the main result of \cite{Willwacher15}.  The Lie algebra $\mathfrak{grt}_1$ is known to contain a free Lie subalgebra with a generator in each odd degree $g \geq 3$ (\cite{Brown12}).  These results let us deduce Theorem~\ref{thm:nonvanishing}.

 \medskip

To the best of our knowledge, the only previously known nonvanishing top weight cohomology group on $\cM_g$ is $\Gr^W_{12}H^6(\cM_3,\Q)$, which has rank 1 by the work of Looijenga mentioned above \cite{Looijenga93}.
Once the general setup of the paper is in place, the result of Looijenga's computation of this top weight cohomology group can be recovered immediately. It corresponds to the $1$-dimensional subspace of graph homology spanned by the complete graph on four vertices.  
Note in general that the top weight cohomology of $\cM_{g}$ is non-tautological and unstable, since stable and tautological classes are of weight equal to their cohomological degree.
Thus, the method presented here probes one piece of the unstable cohomology of $\cM_{g}$ that is especially suited to combinatorial study.

\medskip

The identification of top weight cohomology of $\cM_{g}$ with graph homology, provided by Theorems~\ref{thm:comparison} and Theorem~\ref{thm:gc}, also yields interesting nonvanishing results in degrees other than $4g-6$.  For instance, the nontrivial classes in $H_3(G^{(6)})$, $H_3(G^{(8)})$, and $H_7(G^{(10)})$ discovered by Bar-Natan and McKay \cite{BarNatanMcKay} prove nonvanishing of $H^{15}(\cM_6;\Q)$, $H^{23}(\cM_8;\Q)$, and $H^{27}(\cM_{10};\Q)$.  
It appears that the only previously known example of a nonvanishing odd-degree cohomology group of $\cM_g$ is $H^5(\cM_4; \Q)$ which has rank 1 (and weight 6) by \cite{Tommasi05}.  The interest and difficulty in exhibiting odd cohomology classes on $\cM_g$ was highlighted by Harer and Zagier over three decades ago.  They observed that no such classes were known at the time of their writing, and standard methods could produce classes only in even degree, while their Euler characteristic computations showed that such classes are abundant when $g\gg0$ is even: $(-1)^{g+1} \chi(\cM_g)$ grows like $g^{2g}$.  See \cite[p.~458]{HarerZagier86} and \cite[p.~210]{Harer88}.

\medskip

Finally, we may also use the connection between cohomology of $\cM_{g}$ and graph homology to give an application in the other direction, namely from $\cM_g$ to graph complexes.  Using Harer's computation of the virtual cohomological dimension of $\cM_{g}$ \cite{Harer86} and the vanishing of $H^{4g-5}(\cM_g;\Q)$ \cite{ChurchFarbPutman12, MoritaSakasaiSuzuki13}, we give an independent proof of the following recent result of Willwacher \cite[Theorem~1.1]{Willwacher15}.

\begin{theorem} \label{thm:graphhom}
The graph homology groups $H_k(G^{(g)})$ vanish for $k < 0$.
\end{theorem}

\medskip

Relations between graph (co)homology and (co)homology of moduli spaces of curves were also considered by Kontsevich, but the relationships he studied are conceptually quite different.  For example, he relates genus $g$ curves to genus $2g$ graph homology where we relate genus $g$ curves to genus $g$ graph homology. 
The three different Lie algebras mentioned above correspond to three different types of decorations on graphs, and each comes with a corresponding graph complex that computes homology (or cohomology) of an appropriate moduli space of decorated graphs.
The Lie algebra $a_\infty$ corresponds to graphs decorated with ribbon structure, and moduli spaces of ribbon graphs are homotopy equivalent to moduli spaces of curves with marked points.  This is related to the fact that a punctured Riemann surface deformation retracts to a graph, which remembers a ribbon structure from the deformation.  The cohomology of $\cM_g$ injects into the cohomology of $\cM_{g,1}$, via pullback to the universal curve, and $\cM_{g,1}$ is homotopy equivalent to a moduli space of ribbon graphs of first Betti number $2g$ that bound exactly $1$ open disk.  Forgetting the ribbon structure gives a proper map from this moduli space of ribbon graphs to a moduli space of undecorated graphs. The rational homology of the latter space is computed by the graph complex $G^{(2g)}$ \cite[Section~3]{Kontsevich93}.  

Here, however, we relate the cohomology of $\cM_g$ to the graph complex $G^{(g)}$, not $G^{(2g)}$.  The graphs appear not as deformation retracts of punctured curves, but rather as dual graphs of stable degenerations.  For a (partially conjectural) picture that reproves our main results and relates $G^{(g)}$ to ribbon graph complexes, via hairy oriented graph complexes, see \cite{AWZ20}.

\medskip

The combinatorial structure of $\Delta_g$ and $M_g^\trop$ is intricately related to both the compactification of $\cM_g$ and to the graph complex $G^{(g)}$.  These topological spaces are not strictly needed for proving the main results of this paper; one could instead work in rational chain complexes throughout and define a map directly from $G^{(g)}$ to the top weight row of the weight spectral sequence associated to the Deligne-Mumford stable curves compactification of $\cM_g$. Equivalently, one may relate $G^{(g)}$ to the Feynman transform of the modular operad that associates the vector space $H^0(\overline{\cM}_{g,n};\Q)$, with its trivial $S_n$-action, to each pair $(g,n)$ with $2g -2 + n > 0$, as in \cite[\S6]{AWZ20}.

Nevertheless, the spaces $\Delta_g$ and $M_g^\trop$ provide an intuitive way to visualize and motivate the corresponding constructions with chain complexes.  Moreover, it is well-known in algebraic topology that maps of spaces carry more information than the induced maps of rational chain complexes or homology groups.  In particular the proper map~\eqref{eq:4} should carry more information than the induced map in compactly supported cohomology. 

We note also that our combinatorial topological methods should apply more generally. Any toroidal compactification $\ov{X}\supset X$ of a variety or DM stack gives rise to a combinatorial dual complex whose simple homotopy type is independent of the choice of compactification \cite{boundarycx,Harper17}.  The reduced homology of this dual complex computes the top weight cohomology of $X$, but the space itself encodes more information.  Moreover, if $X$ and $\ov{X}$ are moduli spaces and the universal family of $\ov{X}$ is a toroidal compactification of that of $X$, then the dual complex of $\ov{X}\setminus X$ typically has a natural interpretation as a tropical moduli space.  In this way, one may expect the methods presented here to apply to other moduli spaces, such as moduli of spin curves, moduli of curves with level structure, and moduli of abelian varieties.

\bigskip

\noindent \textbf{Acknowledgments.}  We are grateful to D.~Abramovich, E.~Getzler, M.~Kahle, A.~Kupers, L.~Migliorini, N.~Salter, C.~Simpson, O.~Tommasi, D.~Turchetti, R.~Vakil, and K.~Vogtmann for helpful conversations related to this work.  
MC was supported by NSF
DMS-1204278, DMS-1701924, CAREER DMS-1844768, a Sloan Fellowship and a
Henry Merritt Wriston Fellowship. SG was supported by NSF DMS-1405001 and the European Research Council (ERC) under the European Union's Horizon 2020 research and innovation programme (grant agreement No 682922), by the EliteForsk Prize, and by the Danish National Research Foundation (DNRF92 and DNRF151).  SP was supported by NSF DMS-1702428 and a Simons Fellowship.

\medskip

\section{Graphs, tropical curves, and moduli}  \label{sec:graphs}

In this section, we recall in more detail the construction of the topological
space $\Dg$ as a moduli space for tropical curves, which are marked
weighted graphs with a length assigned to each edge.

\subsection{Weighted graphs and tropical curves} \label{subsec:J-g-n}

Let $G$ be a finite graph, possibly with loops and parallel edges.
All graphs in this paper will be connected.  Write $V(G)$ and $E(G)$
for the vertex set and edge set, respectively, of $G$.  A {\em
  weighted graph} is a connected graph $G$ together with a function
$w\col V(G)\rightarrow \Z_{\ge 0}$, called the {\em weight function}.
The {\em genus} of $(G,w)$ is
\begin{equation*}
  g(G,w)= b_1(G) + \sum_{v\in V(G)}\! w(v),
\end{equation*}
where $b_1(G) = |E(G)|-|V(G)|+1$ is the first Betti number of $G$.

The {\em
  valence} of a vertex $v$ in a weighted graph, denoted
$\val(v)$, is the number of half-edges of $G$ incident to $v$.  In other words, a loop edge based at
$v$ counts twice towards $\val(v)$, once for each end, and an ordinary
edge counts once. We say that $(G,w)$
is {\em stable} if for every $v\in V(G)$, $$2w(v) -2 + \val(v) > 0.$$
For $g \geq 2$, this is equivalent to the condition that every vertex of weight 0
has valence at least 3.

\subsection{The category $\mathbbm{\Gamma}_{g}$}\label{sec:jgn} The connected
stable graphs of genus $g$ form the objects of a
\emph{category} which we denote $\Jg$.  The morphisms in this
category are compositions of contractions of edges $G \rightarrow G/e$
and isomorphisms $G \rightarrow G'$.  For the sake of removing any
ambiguity about what that might mean, let us give a formal and precise definition of $\Jg$.

Formally, then, a graph $G$ is a finite set $X(G) = V(G) \sqcup H(G)$
(of ``vertices'' and ``half-edges''), together with two functions
$s_G, r_G\colon X(G) \to X(G)$ satisfying $s_G^2 = \mathrm{id}$ and
$r_G^2 = r_G$ and that
$$\{x \in X(G) \mid r_G(x) = x\} = \{x \in X(G) \mid s_G(x) = x\} =
V(G).$$ Informally: $s_G$ sends a half-edge to its other half, while
$r_G$ sends a half-edge to its incident vertex.  We let
$E(G) = H(G)/(x \sim s_G(x))$ be the set of edges.  The definition of
 weights, genus, and stability is as before.

The objects of the category $\Jg$ are all connected stable graphs of genus $g$.
For an object $\Gmw = (G,w)$ we shall write $V(\Gmw)$ for $V(G)$
and similarly for $H(\Gmw)$, $E(\Gmw)$, $X(\Gmw)$, $s_\Gmw$ and
$r_\Gmw$.  Then a morphism $\Gmw \to \Gmw'$ is a function $f\colon
X(\Gmw) \to X(\Gmw')$ with the property that
\begin{equation*}
f \circ r_\Gmw = r_{\Gmw'} \circ f \text{ and }f \circ s_\Gmw =
s_{\Gmw'} \circ f,
\end{equation*}
and subject to the following three requirements:  

\begin{itemize}
\item Each $e \in H(\Gmw')$ determines the subset
  $f^{-1}(e) \subset X(\Gmw)$ and we require that it consists of
  precisely one element (which will then automatically be in
  $H(\Gmw)$).
\item Each $v \in V(\Gmw')$ determines a subset
  $S_v = f^{-1}(v) \subset X(\Gmw)$ and
  $\mathbf{S}_v = (S_v,r\vert_{S_v}, s\vert_{S_v})$ is a graph; we
  require that it be connected and have
  $g(\mathbf{S}_v,w\vert_{\mathbf{S}_v}) = w(v)$.
\end{itemize}

\noindent Composition of morphisms $\Gmw \to \Gmw' \to \Gmw''$ in
$\Jg$ is given by the corresponding composition
$X(\Gmw) \to X(\Gmw') \to X(\Gmw'')$ in the category of sets.

Our definition of graphs and the morphisms between them is standard in
the study of moduli spaces of curves and agrees, in essence, with the
definitions in \cite[X.2]{ACG11} and \cite[\S3.2]{acp}, as well as
those in \cite{KontsevichManin94} and \cite{GetzlerKapranov98}.  In particular, our $\mathbbm{\Gamma}_g$ agrees with the category denoted $\Gamma(\!(g,0)\!)$ in \cite{GetzlerKapranov98}.

\begin{remark}
  We also note that any morphism $\Gmw \to \Gmw'$ can be alternatively
  described as an isomorphism following a finite sequence of
  \emph{edge collapses}: for $e \in E(\Gmw)$ there is a morphism
  $\Gmw \to \Gmw/e$ where $\Gmw/e$ is the marked weighted graph
  obtained from $\Gmw$ by collapsing $e$ together with its two
  endpoints to a single vertex $[e] \in \Gmw/e$.  If $e$ is not a
  loop, the weight of $[e]$ is the sum of the weights of the endpoints
  of $e$ and if $e$ is a loop the weight of $[e]$ is one more than the
  old weight of the end-point of $e$.  If
  $S = \{e_1, \dots, e_k\} \subset E(\Gmw)$ there are iterated edge
  collapses $\Gmw \to \Gmw/e_1 \to (\Gmw/e_1)/e_2 \to \dots$ and any
  morphism $\Gmw \to \Gmw'$ can be written as such an iteration
  followed by an isomorphism from the resulting quotient of $\Gmw$ to
  $\Gmw'$.
\end{remark}

We shall say that $\Gmw$ and $\Gmw'$ have the same \emph{combinatorial
  type} if they are isomorphic in $\Jg$.  In fact there are only
finitely many isomorphism classes of objects in $\Jg$, since any
object has at most $6g-6$ half-edges and $2g-2$ vertices; and for
each possible set of vertices and half-edges there are finitely many
ways of gluing them to a graph, and finitely many possibilities for
the weight function.  In order to get a \emph{small}
category $\Jg$ we shall tacitly pick one object in each
isomorphism class and pass to the full subcategory on those objects.
Hence $\Jg$ is a \emph{skeletal} category.  (Although we shall
usually try to use language compatible with any choice of small
equivalent subcategory $\Jg$.)  It is clear that all Hom sets in
$\Jg$ are finite, so $\Jg$ is in fact a finite category.

Replacing $\Jg$ by some choice of skeleton has the effect that if
$\mathbf{G}$ is an object of $\Jg$ and $e \in E(\mathbf{G})$ is an
edge, then the marked weighted graph $\Gmw/e$ is likely not equal to
an object of $\Jg$.  Given $\mathbf{G}$ and $e$, there is a
\emph{morphism} $q \colon \mathbf{G} \to \mathbf{G}'$ in $\Jg$ factoring
through an isomorphism $\Gmw/e \to \Gmw'$.  The pair
$(\mathbf{G}',q)$ is unique up to unique isomorphism (but of course the map $q$ or the isomorphism
$\Gmw/e \to \Gmw'$ on their own need not be unique).  By an abuse of notation, we shall henceforward
write $\mathbf{G}/e \in \Jg$ for the codomain of this unique
morphism, and similarly $G/e$ for its underlying graph.

\begin{definition}\label{def:E-and-H-as-functors}
  Let us define a functor
  \begin{equation*}
    E\colon \Jg^\mathrm{op} \to (\text{Finite sets},
    \text{injections})
  \end{equation*}
  as follows.  On objects, $E(\Gmw) = E(G)$ is the set of edges
  of $\Gmw = (G,w)$ as defined above.  A morphism
  $f\colon \Gmw \to \Gmw'$ determines an injective function $E(f)$
 sending $e' \in E(\Gmw')$ to the
  unique element $e \in E(\Gmw)$ with $f(e) = e'$.
  This clearly preserves composition and identities, and hence defines
  a functor.
  \end{definition}

\subsection{Moduli space of tropical curves}  \label{sec:moduli-space}

We now recall the construction of moduli spaces of stable tropical
curves, as the colimit of a diagram of cones parametrizing possible
lengths of edges for each fixed combinatorial type.  The construction
follows \cite{BrannettiMeloViviani11, Caporaso13}.

Fix an integer $g \geq 2$. A \emph{length function} on
$\Gmw = (G,w) \in \Jg$ is an element
$\ell \in \R_{>0}^{E(\Gmw)}$, and we shall think geometrically of
$\ell(e)$ as the \emph{length} of the edge $e \in E(\Gmw)$.  A genus $g$ \emph{stable tropical curve} is then a pair
$\Gamma = (\Gmw,\ell)$ with $\Gmw \in \Jg$ and
$\ell \in \R_{>0}^{E(\Gmw)}$, and we shall say that $(\Gmw,\ell)$ is
\emph{isometric} to $(\Gmw',\ell')$ if there exists an isomorphism
$\phi\colon \Gmw \to \Gmw'$ in $\Jg$ such that
$\ell' = \ell \circ \phi^{-1}\colon E(\Gmw') \to \R_{>0}$.  The
\emph{volume} of $(\Gmw,\ell)$ is
$\sum_{e \in E(\Gmw)} \ell(e) \in \R_{> 0}$.

We can now describe the underlying set of the topological space
$\Delta_{g}$, which is the main object of study in this paper. It is
the set of isometry classes of genus $g$ stable tropical curves of
volume 1.  We proceed to describe its topology and further structure
as a closed subspace of the moduli space of tropical curves.

\begin{definition}\label{definition:mgn}
  Fix $g \geq 2$.  For each object
  $\Gmw \in \Jg$ define the topological space
  \begin{equation*}
    \sigma(\Gmw)= \R_{\geq 0}^{E(\Gmw)}=\{\ell\colon E(\Gmw) \to \R_{\geq
      0}\}.
  \end{equation*}

  For a morphism $f\colon \Gmw \to \Gmw'$ define the continuous map
  $\sigma f \colon \sigma(\Gmw') \to \sigma(\Gmw)$ by
  $$(\sigma f)(\ell') = \ell\colon E(\Gmw) \to \R_{\geq 0},$$ where
  $\ell$ is given by
  
  \begin{equation*}
    \ell(e) =
    \begin{cases} 
      \ell'(e') & \text{ if $f$ sends $e$ to $e'\in E(\Gmw')$},\\
      0 & \text{ if $f$ collapses $e$ to a vertex}.
    \end{cases}
  \end{equation*}
  This defines a functor
  $\sigma\colon \Jg^\mathrm{op} \to \mathrm{Spaces}$ and the
  topological space $\Mg$ is defined to be the colimit of this
  functor.
\end{definition}

In other words, the topological space $\Mg$ is obtained as
follows. For each morphism $f\colon \Gmw\to\Gmw'$, consider the map
$L_f\colon \sigma(\Gmw') \to \sigma(\Gmw)$ that sends
$\ell'\colon E(\Gmw')\to \R_{>0}$ to the length function
$\ell\colon E(\Gmw)\to \R_{>0}$ obtained from $\ell'$ by extending it
to be 0 on all edges of $\Gmw$ that are collapsed by $f$.  So $L_f$
linearly identifies $\sigma(\Gmw')$ with some face of $\sigma(\Gmw)$,
possibly $\sigma(\Gmw)$ itself.  Then
\begin{equation*}
  \Mg = \left(\coprod \sigma(\Gmw)\right) \Big / \{\ell' \sim L_f(\ell')\},
\end{equation*}
where the equivalence runs over all morphisms $f\colon \Gmw\to\Gmw'$
and all $\ell'\in \sigma(\Gmw')$.

As we shall explain in more detail in Section~\ref{sec:cellular}, $\Mg$ naturally
comes with more structure than a plain topological space; it is an
example of a {\em generalized cone complex}, as defined in
\cite[\S2]{acp}.  This formalizes the observation that $\Mg$ is glued
out of the cones $\sigma(\Gmw)$.

The \emph{volume} defines a function
$v\colon \sigma(\Gmw) \to \R_{\geq 0}$, given explicitly as
$v(\ell) = \sum_{e \in E(\Gmw)} \ell(e)$, and for any morphism
$\Gmw \to \Gmw'$ in $\Jg$ the induced map
$\sigma(\Gmw') \to \sigma(\Gmw)$ preserves volume.  Hence there is an
induced map $v\colon \Mg \to \R_{\geq 0}$, and there is a unique element
in $\Mg$ with volume 0 which we shall denote $\bullet_{g}$.  The
underlying graph of $\bullet_{g}$ consists of a single vertex with
weight $g$.

\begin{definition}\label{def:dgn}
  We let $\Dg$ be the subspace of $\Mg$ parametrizing curves of
  volume 1, i.e.,\ the inverse image of $1 \in \R$ under
  $v: \Mg \to \R_{\ge 0}$.
\end{definition}

\noindent Thus $\Dg$ is homeomorphic to the link of $\Mg$ at the
cone point $\bullet_{g}$.  Moreover, it inherits the structure of a
symmetric $\Delta$-complex, as we shall define in
Section~\ref{sec:cellular}, from the generalized cone complex structure on
$\Mg$.  See Remark~\ref{rem:cones}. 
 
\subsection{Kontsevich's graph complex}
\label{sec:konts-graph-compl}

Let us briefly recall the (commutative) graph complex, defined by Kontsevich in \cite{Kontsevich93}.  This chain complex comes in two versions, differing by some important signs.  Kontsevich's original paper is mostly focused on what he calls the ``even'' version of the graph complex; it is related to invariants of odd-dimensional manifolds and by Willwacher's results to deformations of the operad $e_n$ for odd $n$.  This is the same version as considered by e.g. \cite{ConantVogtmann03} and \cite{ConantGerlitsVogtmann05}.  The other version, called ``odd'' in \cite{Kontsevich93}, is related to invariants of even-dimensional manifolds, deformations of the operad $e_n$ for even $n$, and by the main theorem of \cite{Willwacher15} to the Grothendieck--Teichm\"uller Lie algebra.  It is the latter version which is relevant to our paper and shall be recalled here.  Both are considered in \cite{BarNatanMcKay} where they are called the ``fundamental example'' and the ``basic example'' of graph homology, respectively (the assertion of op.cit.\ that the basic example does not occur in nature is of course no longer true).

The graph complex is defined by letting $G^{(g)}$ be the rational vector space generated by $[\Gamma,\omega]$ where $\Gamma$ is a connected graph of genus $g$ all of whose vertices have valence at least 3, in other words an object of $\mathbbm{\Gamma}_g$ in which $w(v) = 0$ for all $v \in V(\Gamma)$.  The ``orientation'' $\omega$ is a total ordering on $E(\Gamma)$,
and this notation is subject to the relation $[\Gamma,\omega] = \mathrm{sgn}(\sigma)[\Gamma',\omega']$ if there exists an isomorphism of graphs $\Gamma \cong \Gamma'$ under which the total orderings are related by a permutation $\sigma$.  It follows from this relation that $[\Gamma,\omega] = 0$ if $\Gamma$ has at least two parallel edges, since then there is an automorphism of $\Gamma$ inducing an odd permutation of its edge set.  The boundary map in this chain complex is induced by
\begin{equation}\label{eq:1}
  \partial [\Gamma,\omega] = \sum_{i=0}^n (-1)^i [\Gamma/e_i,\omega\vert_{E(\Gamma/e_i)}],
\end{equation}
where $\omega = (e_0 < e_1 < \dots < e_n)$ is the total ordering on the set $E(\Gamma)$ of edges of $\Gamma$, the graph $\Gamma/e_i$ is the result of collapsing $e_i \subset \Gamma$ to a point, and $\omega\vert_{E(\Gamma/e_i)}$ is the induced ordering on the subset $E(\Gamma/e_i) = E(\Gamma) \setminus \{e_i\} \subset E(\Gamma)$.  In case $e_i \in E(\Gamma)$ is a loop, we interpret the corresponding term in~\eqref{eq:1} as zero.
\begin{example}
  For $g\ge 3$, let $W_g \in G^{(g)}$ be the ``wheel graph'' with $g$ trivalent vertices, one $g$-valent vertex, and $2g$ edges arranged in a wheel shape with $g$ spokes, and with some chosen ordering of its edge set.  The graph underlying $W_5$ is depicted in Figure~\ref{fig:wheel}.  Then $\partial W_g = 0$.  This gives a non-zero cycle for odd $g$, which we also denote $W_g$.
\end{example}

\begin{figure}[h]
\begin{tikzpicture}[my_node/.style={fill, circle, inner sep=1.75pt}, scale=1]
\def\R{.587} 
\def\H{.2} 
\def\Ratio{1.42} 
\def\M{1.2}
\begin{scope}[scale=1]
\node[my_node] (1) at (0,1){};
\node[my_node] (5) at (.95,.31){};
\node[my_node] (3) at (0.59,-.81){};
\node[my_node] (4) at (-.59,-.81){};
\node[my_node] (2) at (-.95,.31){};
\node[my_node] (W) at (0,0){};
\draw[ultra thick] (1)--(W)--(2)--(W)--(3)--(W)--(4)--(W)--(5);
\draw[ultra thick] (1)--(5)--(3)--(4)--(2)--(1);
\end{scope}
\end{tikzpicture}
\caption{The graph $W_5$.}\label{fig:wheel}
\end{figure}
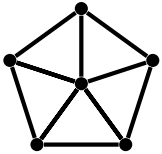

\noindent
Indeed, any contraction of a single edge $e$, spoke or non-spoke, leads to a graph $W_g/e$ with two parallel edges, which then represents the zero element in the graph complex.

  The automorphism group of $W_g$ is isomorphic to $S_4$ when $g=3$ and is isomorphic to the dihedral group $D_{2g}$ when $g>3$, and it is easy to verify that it acts by even permutations on $E(W_g)$ when $g$ is odd.  Hence $W_g \neq 0 \in G^{(g)}$ for odd $g$.  (Notice that so far we are only making the elementary claim that it is non-zero on the chain level, although it in fact turns out to represent a non-zero homology class.)  On the other hand, the involutions in the dihedral group act by odd permutations on $E(W_g)$ for even $g$, and hence $W_g = 0$ in this case.

Grading conventions differ from author to author.
In Kontsevich's original paper, the grading of this chain complex is by number of vertices $|V(\Gamma)|$.  We shall instead use conventions better suited for comparison with \cite{Willwacher15}, in which the degree of $\Gamma$ is $|V(\Gamma)| - (g+1)$.  In this grading the wheel graph has degree 0.  As we shall see later, it also has the effect of making $G^{(g)}$ a connective chain complex, i.e., its homology vanishes in negative degrees.
Willwacher's paper \cite{Willwacher15} considers the linearly dual cochain complex which he denotes $\mathsf{GC}$ or $\mathsf{GC}_2$, so that
\begin{equation*}
  \mathsf{GC} = \prod_{g = 2}^\infty \Hom(G^{(g)}, \Q),
\end{equation*}
where $\Hom(-,\Q)$ denotes the graded dual.  For example, for odd $g$ there is a cochain $W_g^\vee$ given by sending the wheel graph $(W_g,\omega) \mapsto \pm 1$ (sign depending on $\omega$) and any other graph to 0. 
In Willwacher's grading convention, the differential on $\mathsf{GC}$ raises the degree by $1$, the cohomological degree of (the dual basis element corresponding to) $[\Gamma,\omega]$ being  $|V(\Gamma)| - (g+1) = |E(\Gamma)| - 2g.$  
The differential on $\mathsf{GC}$ is then given by precomposing with (\ref{eq:1}).  

As explained in (\cite[Proposition 3.4]{Willwacher15}), $\mathsf{GC}$ carries a natural combinatorially defined Lie bracket, making it a differential graded Lie algebra.
The main result of Willwacher's paper gives an isomorphism between the Grothendieck--Teichm\"uller Lie algebra and graph cohomology in degree 0
\begin{equation*}
  H^0(\mathsf{GC}) \cong \mathfrak{grt}_1.
\end{equation*}
A connected genus $g$ graph gives a degree-0 cochain if it has precisely $2g$ edges (and hence $g+1$ vertices).  Any element of $H^0(\mathsf{GC})$ may be evaluated on the cycle $W_g$.  The dual basis element $W_g^\vee$ has cohomological degree 0 in $\mathsf{GC}$, but is likely not a cocycle.
By definition, the Lie algebra $\mathfrak{grt}_1$ consists of elements $\phi$ of the completed free Lie algebra on two elements, satisfying certain explicit equations which we shall not recall (see \cite[\S 6.1]{Willwacher15}).
An important consequence of this isomorphism is the following.
\begin{theorem}[\cite{Willwacher15}]\label{thm:wheel-nonzero}
  For any odd $g \geq 3$ there exists an element $\sigma_g \in H^0(\mathsf{GC})$ with $\langle\sigma_g,W_g \rangle \neq 0$.  Hence $[W_g] \neq 0 \in H_0(G^{(g)})$, i.e., the wheel cycle $W_g$ is not a boundary.
\end{theorem}
\begin{proof}[Proof sketch]
  Starting from a suitable Drinfeld associator, Willwacher in \cite[\S9]{Willwacher15} translates the corresponding element $\sigma_g \in \mathfrak{grt}_1$ into $\mathsf{GC}$ and proves that the resulting cocycle in $\mathsf{GC}$ has non-zero coefficient of $W_g^\vee$.
\end{proof}
 
\begin{theorem}\label{thm:nonvanishing-graph-homology}
The group $H_0(G^{(g)})$ is nonzero for $g =3$, $g = 5$, and $g \geq 7$.  
Moreover, $\dim H_0(G^{(g)})$ grows at least exponentially.  More precisely, $$\dim H_0(G^{(g)}) > \beta^g + \text{constant}$$ for any $\beta< \beta_0$, where $\beta_0 \approx 1.3247\ldots$ is the real root of $t^3-t-1=0$ . 
\end{theorem}

\begin{proof}
Let $V$ denote the graded $\Q$-vector space generated by symbols $\sigma_{2i+1}$ in degree $2i+1$ for each $i\ge 1$, and let $\mathrm{Lie}(V)$ be the free Lie algebra on $V$.  
As explained in \cite{Willwacher15}, the result of \cite{Brown12} implies that the classes $\sigma_{2i+1} \in \mathfrak{grt}_1$  
together with the Lie algebra structure on $\mathfrak{grt}_1$ gives rise to an injection 
\begin{equation}\label{eq:lie-inj}
\mathrm{Lie}(V)\hookrightarrow \mathfrak{grt}_1 \cong H^0(\mathsf{GC})\cong {\bigg( \bigoplus_{g\ge2} H_0 (G^{(g)}) \bigg)}^{\!\!\raisebox{-2pt}{$\scriptstyle\vee$}}.
\end{equation} Thus $\sigma_3,\sigma_5\ne 0\in \mathfrak{grt}_1$, and since any even number $g \geq 8$ may be written as $g = 3 + (g-3)$ with $g-3 > 3$, we also have $[\sigma_3,\sigma_{g-3}] \neq 0 \in \mathfrak{grt}_1$, which gives rise to a non-zero homomorphism $H_0(G^{(g)}) \to \Q$. More specifically, for $g\ge 8$ even, $H_0(G^{(g)})$ contains a non-zero homology class whose Lie cobracket contains a term $W_3 \otimes W_{g-3}$.

For the asymptotic statement, we shall compute the Poincar\'e series (i.e., the generating function for dimension of graded pieces) of $\mathrm{Lie}(V)$, using a variant of Witt's formula for the dimension of the graded pieces of a free, finitely generated Lie algebra that is generated in degree 1, and then appeal to~\eqref{eq:lie-inj}.  The Poincar\'e series of $V$ is $f(t) = t^3/(1-t^2)$. The universal enveloping algebra $U(\mathrm{Lie}(V))$ is isomorphic to the free associative algebra $\bigoplus_{n\ge 0}V^{\otimes n}$, so has Poincar\'e series $1/(1-f)$. Now let $S(\mathrm{Lie}(V))$ denote the free commutative $\Q$-algebra on $\mathrm{Lie}(V)$; it has Poincar\'e series $$\prod_{d\ge 0} \frac{1}{(1-t^d)^{A_d}},$$
where $A_d:=\dim \mathrm{Lie}(V)_d$ are the sought-after coefficients of the Poincar\'e series for $\mathrm{Lie}(V)$.  The Poincar\'e-Birkhoff-Witt theorem implies that $U(\mathrm{Lie}(V)) \cong S(\mathrm{Lie}(V))$ as graded vector spaces, so $1/(1-f) = \prod_{n\ge 0} {1}/{(1-t^n)^{A_n}}$. Applying $t \tfrac{d}{dt} \operatorname{log}(\cdot)$ to both sides yields
\begin{equation}\label{eq:after-log}
p(t):=\frac{t^3(3-t^2)}{(1-t^2)(1-t^2-t^3)} = \sum_{d\ge 0} dA_d \,\frac{t^d}{1-t^d}.
\end{equation}
Write $p(t) = \sum_{n\ge0} a_n t^n$. To analyze the $a_n$, notice that $p(t)$ has five simple poles, at the roots of $(1-t^2)(1-t^2-t^3) = 0$.  There is a unique root $\alpha\approx 0.75488\ldots$ having smallest absolute value, and $\operatorname{Res}_\alpha p(t) = -\alpha$ (the exact value of the residue is not important).  Therefore $p(t) = -\alpha/(t-\alpha) + \sum_{n\ge0} b_n t^n = \sum_{n\ge 0} (\tfrac{1}{\alpha^n} + b_n)t^n$, where  $\sum_{n\ge0} b_n t^n$ converges on a disc centered at $0$ of radius $>\alpha$. Therefore $b_n\alpha^n\to 0$ and $a_n\alpha^n = (\tfrac{1}{\alpha^n} + b_n)\alpha^n \to 1$. Setting $\beta_0 = 1/\alpha\approx 1.32472\ldots$, then $a_n\to \beta_0^{n}$.  

Now equating coefficients in~\eqref{eq:after-log} yields $a_n = \sum_{d|n} dA_d$, so $$A_n = \frac{1}{n} \sum_{d|n} \mu\!\left(\frac{n}{d}\right) a_d$$ by M\"obius inversion. Since $a_n$ grows exponentially, the summand when $d=n$, namely $a_n/n$, eventually dominates the other terms in the sum, and $A_n$ grows faster than $\beta^n$ for any $\beta<\beta_0$.
\end{proof}

\begin{remark}
  While the exponential growth rate in Theorem~\ref{thm:nonvanishing-graph-homology} relies crucially on the main results of \cite{Willwacher15} and \cite{Brown12}, there are easier and more direct proofs of Theorem~\ref{thm:wheel-nonzero}, constructing an element of $H^0(\mathsf{GC})$ detecting $[W_g]$ using configuration space integrals as in \cite[\S6]{RossiWillwacher14}.  This nonvanishing of $[W_g]$ for odd $g \geq 3$ is sufficient for  disproving \cite[Conjecture~7C]{Kontsevich93} and \cite[Conjecture~9]{ChurchFarbPutman14}.
\end{remark}

\section{Symmetric $\Delta$-complexes}
\label{sec:symm-semis-objects}
\label{sec:cellular}

\begin{definition}
For $p\ge -1$ an integer, we set $$[p] = \{0,\ldots,p\}.$$ \end{definition}
\noindent This notation includes $[-1] = \emptyset$ by convention.

As usual we write $\Delta^p \subset \R^{p+1}$ for the standard simplex, i.e., the convex hull of the standard basis vectors $e_0, \dots, e_p$; its points are $t = (t_0, \dots, t_p) = \sum t_i e_i$ with $t_i \geq 0$ and $\sum t_i = 1$.
Associating the standard simplex
$\Delta^p$ to the number $p$ may be promoted to a functor from finite sets to
topological spaces; for a finite set $S$ define
$\Delta^S = \{a\colon S \to [0,\infty) \mid \sum a(s) = 1\}$ in the
Euclidean topology and for any map of finite sets $\theta\colon S \to T$,
define $\theta_*\colon \Delta^S \to \Delta^T$ by
\begin{equation*}
  (\theta_* a)(t) = \sum_{\theta(s) = t} a(s).
\end{equation*}
The usual $p$-simplex is recovered as $\Delta^p = \Delta^{[p]}$ with
$[p] = \{0, \dots, p\}$.

\subsection{Recollections on $\Delta$-complexes}
\label{sec:recoll-delta-compl}

Let us write $e_i \in \Delta^p$ for the $i$th vertex, $0 \leq i \leq p$.  We order the set of vertices in $\Delta^p$ as $e_0 < \dots < e_p$. 
Let $\Delta_\inj$ be the category with one object
$[p] = \{0, \dots, p\}$ for each integer $p \geq 0$, in which the
morphisms $[p] \to [q]$ are the order preserving injective maps.  We
shall take the following as the official definition of a
$\Delta$-complex (sometimes known as ``semi-simplicial set'' in the more recent literature). \begin{definition}
  A $\Delta$-complex
   is a functor $X\colon \Delta_\inj^\mathrm{op} \to \Sets$.
\end{definition}

The \emph{geometric realization} of a $\Delta$-complex $X$ is
\begin{equation}\label{eq:realization}
  |X| = \Big(\coprod_{p = 0}^\infty X([p]) \times \Delta^p\Big) \big / \sim,
\end{equation}
where $\sim$ is the equivalence relation generated by $(x,\theta_* a) \sim (\theta^* x, a)$ for $x \in X([q])$, $\theta\colon [p] \to [q]$ in $\Delta_\inj$, and $a \in \Delta^p$.  Each element $x \in X([p])$ determines a map of topological spaces $x\colon \Delta^p \to |X|$, and the functor $X\colon \Delta_\inj^\mathrm{op} \to \Sets$ may be recovered from the topological space $|X|$ together with this set of maps from simplices.  

As is customary, we shall usually write $X_p = X([p])$.  
 We also write $\delta^i\colon [p-1] \to [p]$ for the unique order preserving injective map whose image does not contain $i$, and $d_i\colon X_p \to X_{p-1}$ for the induced map.
There is also a category of \emph{augmented} $\Delta$-complexes, which are functors $(\Delta_\inj \cup\{[-1]\})^\mathrm{op} \to \Sets$, where $[-1] = \emptyset$ is added to $\Delta_\inj$ as initial object.  The geometric realization $|X|$ then comes with a continuous map $\epsilon\colon |X| \to X_{-1}$.  

\subsection{Symmetric $\Delta$-complexes}
\label{sec:symm-delta-compl}
\label{sec:generalized}
\label{sec:symmetric}

The notion of $\Delta$-complexes may be generalized to
allow, in their geometric realizations, gluing along maps $\Delta^q \to \Delta^p$ that do not preserve the ordering of the vertices. This includes gluing along maps from $\Delta^p$ to itself induced by permuting the vertices.  
\begin{definition}
  Let $I$ be the category with the same objects as
  $\Delta_\inj \cup \{[-1]\}$, but whose morphisms $[p] \to [q]$ are
  all injective maps $\{0, \dots, p\} \to \{0, \dots, q\}$.  A
  \emph{symmetric $\Delta$-complex} (or symmetric semi-simplicial set) is a functor
  $X\colon I^\mathrm{op} \to \Sets$.
\end{definition}
Such a functor is given by a set $X_p$ for each
$p \geq -1$, actions of the symmetric group $S_{p+1}$ on $X_p$ for all $p$, and face maps $d_i\colon X_p \to X_{p-1}$ for
$0 \leq i \leq p$.  The face maps satisfy the usual simplicial
identities as well as a compatibility with the symmetric group action.
We have chosen the name in analogy with the
``symmetric simplicial sets'' in the literature (e.g., \cite{Grandis01}), which is a similar notion also including degeneracy maps.  
The \emph{geometric realization} of $X$ is given by
formula~(\ref{eq:realization}), where the equivalence relation now
uses all morphisms $\theta$ in $I$. 

Symmetric $\Delta$-complexes also come with a set $X_{-1} = X(\emptyset)$ and there is an augmentation map $|X| \to X_{-1}$.  (So strictly speaking ``augmented symmetric $\Delta$-complexes'' would be a more accurate name, but we use ``symmetric $\Delta$-complexes'' for brevity.) 

The \emph{standard orthant} $\R_{\geq 0}^{[p]} = \prod_{i=0}^p [0,\infty)$ is functorial in $[p] \in I$ by letting $\theta \in I([p],[q])$ act as $\theta_*(t_0, \dots, t_p) = \sum t_i e_{\theta(i)}$, where $e_i \in \R^{[p]}$ denotes the $i$th standard basis vector.  Replacing $\Delta^p$ by the standard orthant in the definition of $|X|$ we arrive at the \emph{cone over} $X$:
\begin{equation}\label{eq:realization-cone}
  CX = \Big(\coprod_{p = {-1}}^\infty X_p \times \R_{\geq 0}^{[p]}\Big) \big / \sim,
\end{equation}
where  $\sim$ is the equivalence relation generated by $(x,\theta_* a) \sim (\theta^* x, a)$ for $p,q \geq -1$, $x \in X_q$, $a \in \R_{\geq 0}^{[p]}$, and $\theta \in I([p],[q])$.  The maps $\ell\colon \R_{\geq 0}^{[p]} \to \R$ given by $(t_0, \dots, t_p) \mapsto \sum t_i$ are compatible with this gluing, and induce a canonical map
\begin{equation*}
  \ell_X\colon CX \to \R_{\geq 0},
\end{equation*}
so that $X \mapsto CX$ naturally takes values in the category of spaces over $\R_{\geq 0}$.  We have canonical homeomorphisms $\ell_X^{-1}(1) = |X|$ and $\ell_X^{-1}(0) = X_{-1}$, and from $\ell_X^{-1}([0,1])$ to the mapping cone of the augmentation $|X| \to X_{-1}$.  The inclusions $X_{-1} \subset \ell_X^{-1}([0,1]) \subset CX$ are both deformation retractions.  It follows that the quotient $CX/|X|$ deformation retracts to the mapping cone of $|X| \to X_{-1}$.  When $X_{-1} = \{\ast\}$ the space $CX$ deformation retracts to the cone over $|X|$ (hence the name) and $CX/|X|$ deformation retracts to the unreduced suspension $S|X|$.

\begin{example}
  The representable functor $I(-,[p])\colon I^\mathrm{op} \to \Sets$ has geometric realization $|I(-,[p])| \cong \Delta^p$.
\end{example}

\begin{example}
  An (abstract) simplicial complex $K$ with vertex set $V$ determines
  a symmetric $\Delta$-complex $X_K\colon I^\mathrm{op} \to \Sets$, sending $[p]$ to the set of injective maps $f\colon [p] \to V$ whose image spans a simplex of $K$.  The realizations of $K$ as a simplicial complex and $X_K$ as a symmetric $\Delta$-complex are canonically homeomorphic.
\end{example}

\begin{example}\label{ex:halfinterval}
  A typical example of a symmetric $\Delta$-complex in which the symmetric groups do not act freely is the {\em half interval} given as a coequalizer of the two distinct morphisms
  $$X = \mathrm{colim} (I(-,[1]) \double I(-,[1])),$$ where $X_{-1}$, $X_0$, and $X_1$ are one-element
  sets and $X_p = \emptyset$ for $p \geq 2$.  The unique element in
  $X_1$ gives a map $\Delta^1 \to |X|$ which is not injective; it
  identifies $|X|$ with the topological quotient of $\Delta^1$ by the
  action of $\Z/2\Z$ that reverses the orientation of the interval.
\end{example}

\subsection{Cellular chains}
\label{sec:cellular-chains}

The rational homology of $|X|$ and the relative homology of $(CX,|X|)$ may be calculated by cellular chain complexes, functorially associated to $X: I^\mathrm{op} \to \Sets$.

\begin{definition}\label{def:cellular-chains-set}
  Let $R$ be a commutative ring and write $R X_p$ for the free $R$-module spanned by the set $X_p$.  The group of cellular $p$-chains $C_p(X;R)$ is
  \begin{equation*}
    C_p(X;R) = (R^\mathrm{sgn} \otimes_{R S_{p+1}}  R X_p)
  \end{equation*}
  where $R^\mathrm{sgn}$ denotes the action of $S_{p+1}$ on $R$ via the sign.
\end{definition}
The boundary map $\partial: C_p(X) \to C_{p-1}(X)$ is the unique map that makes the following diagram commute:
\begin{equation*}
  \xymatrix{
    {R X_p} \ar[rr]^-{\sum (-1)^i (d_i)_*} \ar@{->>}[d] && {R X_{p-1}}\ar@{->>}[d]^\pi\\
    {C_p(X;R)} \ar[rr]^-{\partial} &&  C_{p-1}(X;R).
  }
\end{equation*}
It is easily verified that such a homomorphism $\partial$ exists, and satisfies $\partial^2 = 0$.

Similarly, we define cochains 
\begin{equation*}
  C^p(X;R) = \Hom_\Z(C_p(X;\Z),R) = \Hom_{R S_{p+1}}(R X_p,R^\mathrm{sgn}),
\end{equation*}
with coboundary $\delta = (-1)^{p+1} \partial^\vee\colon C^p(X;R) \to C^{p+1}(X;R)$.
In other words, $C^p(X;R)$ is the $R$-module consisting
of all set maps $\phi\colon X_p \to R$ which satisfy
$\phi (\sigma x) = \mathrm{sgn}(\sigma) \phi(x)$ for all
$x \in X_p$ and all $\sigma \in S_{p+1}$.

\newcommand{\sing}{\mathrm{sing}}

To compare this cellular chain complex to singular homology, write $\iota_p \in C_p^\sing(\Delta^p)$ for the chain given by the identity map of $\Delta^p$, and $\iota_p'\in C_p^\sing(\Delta^p)$ for its barycentric subdivision (in the sense of e.g.\ \cite[p.\ 122]{Hatcher02} or \cite[\S IV.17]{Bredon13}).  Any element $x \in X_p$ gives a map $x\colon \Delta^p \to |X|$, and we define a natural transformation
\begin{equation}\label{eq:12}
  \begin{aligned}
  C_p(X;\Z) &\to C_p^\mathrm{sing}(|X|;\Z)\\
   x &\mapsto x_*(\iota_p'),
\end{aligned}
\end{equation}
This is well defined because $\iota'_p$ satisfies $\sigma_*(\iota_p') = \mathrm{sgn}(\sigma) \iota_p'$ for all $\sigma \in S_{p+1}$ and a defines a chain homomorphism $C_*(X;\Z)_{\geq 0} \to C_*^\mathrm{sing}(|X|;\Z)$ where we write $C_*(X;\Z)_{\geq 0}$ for the quotient by $\Z X_{-1}$.  By applying $R \otimes_\Z(-)$ or $\Hom_\Z(-,R)$ to~(\ref{eq:12}) we obtain natural natural transformations of chains and cochains with coefficients in $R$.

\begin{proposition}\label{prop:cellularhomology}
  The homomorphisms
  \begin{align*}
    H_p(C_*(X;R)_{\geq 0},\partial) &\to H_p^\mathrm{sing}(|X|;R)\\
    H^p(C^*(X;R)_{\geq 0},\delta) &\leftarrow H^p_\mathrm{sing}(|X|;R)
  \end{align*}
  induced by the (co)chain homomorphisms defined above are isomorphisms, provided
  the orders of stabilizers of $S_{p+1}$ on $X_p$ are invertible in $R$ (e.g., if the actions are all free or if $\Q \subset R$).  Under this assumption, there are induced isomorphisms
  \begin{align*}
    H_p(C_*(X;R),\partial) & \cong H^\mathrm{sing}_{p+1}(CX,|X|;R)\\
    H^p(C_*(X;R),\delta) & \cong H_\mathrm{sing}^{p+1}(CX,|X|;R).
  \end{align*}
  When $X_{-1}$ is a singleton, the right hand sides here are reduced homology and cohomology $\widetilde H_p^\sing(|X|;R)$ and $\widetilde H^p_\sing(|X|;R)$.
\end{proposition}
\begin{proof}
  The symmetric $\Delta$-complex $X$ is filtered by subcomplexes $X^{(p)} \subset X$ defined by setting $X^{(p)}_q = X_q$ for $q \leq p$ and $X^{(p)}_q = \emptyset$ for $q > p$.  The quotient space $|X^{(p)}|/|X^{(p-1)}|$ may be identified with the orbit space
  \begin{equation*}
    |X^{(p)}|/|X^{(p-1)}| \cong \bigg(\frac{X_p \times \Delta^p}{X_p \times \partial \Delta^p}\bigg) / S_{p+1},
  \end{equation*}
  and the induced map
  \begin{equation*}
    R^\mathrm{sgn} \otimes_{RS_{p+1}} RX_p \to H^\sing_p((X_p\times \Delta^p)/S_{p+1},(X_p \times \partial \Delta^p)/S_{p+1};R)
  \end{equation*}
  is an isomorphism under the assumption.  Now proceed by induction on skeleta, using the five-lemma and the long exact sequences associated to the pairs $(X^{(p)},X^{(p-1)})$, exactly as in the proof of \cite[Theorem 2.2.27]{Hatcher02}.  

  For the augmented statement use $|X| \to X_{-1}$ to add one more term to the singular chain complex
  \begin{equation*}
    \dots \to C^\sing_1(|X|;R) \to C^\sing_0(|X|;R) \to RX_{-1} \to 0
  \end{equation*}
  This complex calculates $H^\sing_{*+1}(CX,|X|;R)$, 
   since the inclusion $X_{-1} \subset CX$ is a deformation retraction.  The claim is now easily deduced from the absolute case, and cohomology is similar.
\end{proof}
Henceforth we shall use the same notation $H_*(-;R)$ and $H^*(-;R)$ for the singular and cellular theories.

\begin{definition}
  For a symmetric $\Delta$-complex $X$ define
  \begin{align*}
    H_p(X;R) = H_p(C_*(X;R),\partial)\\
    H^p(X;R) = H^p(C^*(X;R),\delta).
  \end{align*}
  When $X_{-1}$ is a singleton these agree with $\widetilde H_p(|X|;R)$ and $\widetilde H^p(|X|;R)$ respectively, provided orders of stabilizers of $S_{p+1}$ on $X_p$ are invertible in $R$.
\end{definition}

In practice the cellular homology may be studied using the following observation.
\begin{lemma}
  Let $X$ be a symmetric $\Delta$-complex, and let $R$ be a commutative ring.  Suppose we are given subsets $T_p \subset X_p$ for some $p$, and suppose that either
  \begin{itemize}
  \item the induced map $S_{p+1} \times T_p \to X_p$ is a bijection for all $p$, or
  \item $\Q \subset R$, the composition $T_p \to X_p \to X_p/S_{p+1}$ is injective, the stabilizer of any $x \in T_p$ is contained in $A_{p+1}$, and any point $x \in X_p$ whose stabilizer is contained in $A_{p+1}$ is in the $S_{p+1}$-orbit of some $x' \in T_p$.
  \end{itemize}
  Then the map
  \begin{equation*}
    RT_p \to C_p(X;R)
  \end{equation*}
  is an isomorphism of $R$-modules.\qed
\end{lemma}

\noindent Thus $H_*(|X|;\Q)$ may be calculated from a chain complex with one generator for each element in a set of representatives for orbits of elements with alternating stabilizers.

\begin{remark}
  A similar construction was used in \cite{HatcherVogtmann98} to find
  a small model for the rational chains of a certain space, except
  that instead of our $\Delta^{p}/H$ for $H < S_{p+1}$ their basic
  building blocks are of the form $[0,1]^n/H$ for certain subgroups
  $H$ of the symmetry group of a cube.
  There is also a general construction used by \cite{Berkovich99}, 
   in which simplices are replaced by polysimplicial sets.
\end{remark}

\subsection{Colimit presentations and subdivision}
\label{sec:subdivision}

Any symmetric $\Delta$-complex is isomorphic to the colimit of a diagram consisting of representable functors $I(-,[p]) \colon I^\mathrm{op} \to \Sets$ and morphisms between them.  This is a special case of a general fact about presheaves of sets on a small category, cf.\ \cite[\S III.7]{MacLane98}, but let us recall how it works in our case.

Given $X \colon I^\mathrm{op} \to \Sets$, define a category $\JX$ whose objects are pairs $([p],x)$ with $x \in X([p])$ and whose morphisms $([p],x) \to ([p'],x')$ are the $\theta \in I([p'],[p])$ with $X(\theta)(x) = x'$.  For later use we point out that $\theta$ is an isomorphism in $\JX$ if and only if $p = p'$.  
Note that there is a canonical morphism of symmetric $\Delta$-complexes
\begin{equation}
  \label{eq:7}
  \colim_{([p],x) \in \JX} I(-,[p]) \to X,
\end{equation}
assembled from the morphisms $x\colon I(-,[p]) \to X$.  Using that colimits in the category of symmetric $\Delta$-complexes are calculated object-wise, it is easy to verify that this morphism is in fact always an isomorphism.

We sometimes use this to reduce a statement about all symmetric $\Delta$-complexes to a statement about representable ones, when the statement is preserved by taking colimits.

\begin{lemma}\label{lem:geometric-realization-colimit}
  The functor $X \mapsto |X|$ from symmetric $\Delta$-complexes to topological spaces preserves all small colimits.
\end{lemma}
By~\eqref{eq:7}, the geometric realization functor is characterized by Lemma~\ref{lem:geometric-realization-colimit} and the homeomorphisms $|I(-,[p])| = \Delta^p$, natural in $[p] \in I$.  Similarly for $X \mapsto CX$.

\begin{proof}
  For a topological space $Z$ we let $\mathrm{Sing}(Z)$ be the
  symmetric $\Delta$-complex which sends $[p]$ to the set of all
  continuous maps $\Delta^p \to Z$.  It is easily verified that the resulting functor $\mathrm{Sing}$ is right adjoint to the geometric realization functor, which therefore preserves all small colimits (\cite[V.5]{MacLane98}).
\end{proof}

Let us briefly discuss a barycentric subdivision functor, from symmetric $\Delta$-complexes to (augmented) $\Delta$-complexes.
We first define the barycentric subdivision of the symmetric $\Delta$-complex $I(-,[p])$ by sending $[q] \in \Delta_\mathrm{inj} \cup \{[-1]\}$ to the set of all flags of subsets $(\emptyset \subsetneq A_0 \subsetneq \dots \subsetneq A_q \subseteq [p])$.  The subdivision $\mathrm{sd}(X)$ of a general $X\colon I^\mathrm{op} \to \Sets$ is then
defined as the colimit in augmented $\Delta$-complexes
\begin{equation*}
  \mathrm{sd}(X) = \colim_{([p],x) \in \JX} \mathrm{sd}(I(-,[p])).
\end{equation*}
Explicitly, this spells out to the formula
\begin{equation*}
  \mathrm{sd}(X)([q]) = \big(\coprod_p X_p \times \mathrm{sd}(I(-,[p]))_q\big)/\sim,
\end{equation*}
where $\sim$ is the equivalence relation generated by
$(x,\theta^*b)\sim (X(\theta)(x),b)$ whenever $x\in X_p$, $b\in \mathrm{sd}(I(-,[p']))_q,$ and $\theta\in I([p'],[p])$.

Equivalently, $\mathrm{sd}(X)([q])$ may be explicitly described as the set of equivalence classes of conservative functors $\sigma\colon (0 < \dots < q) \to \JX^\mathrm{op}$, up to natural isomorphism of functors.  In any case, let us emphasize that the subdivision of a symmetric $\Delta$-complex is an (augmented) ordinary $\Delta$-complex, not merely a symmetric one.

\begin{lemma}
  The geometric realizations of a symmetric $\Delta$-complex $X$ and the $\Delta$-complex $\mathrm{sd}(X)$ are canonically homeomorphic.
\end{lemma}
\begin{proof}
  Since both geometric realization and barycentric subdivision preserve all small colimits, it suffices to construct a natural homeomorphism
  \begin{equation*}
    |I(-,[p])| \cong |\mathrm{sd}(I(-,[p]))|,
  \end{equation*}
  which is done in the usual way: the left hand side is $\Delta^p$, a non-empty subset $A \subset [p]$ determines a face of $\Delta^p$, and the corresponding vertex on the right hand side is sent to the barycenter of that face; extend to an affine map on each simplex.
\end{proof}

\begin{remark}
Colimit presentations may be used to make many other definitions, or illuminate old ones. For example, the \emph{join} $X \ast Y$ of two symmetric $\Delta$-complexes $X$ and $Y$ may be defined by requiring $(I(-,[p])) \ast (I(-,[q])) = I(-,[p] \amalg [q])$ and requiring $X \ast Y$ to preserve colimits in $X$ and $Y$ separately.  The chains functor $X \mapsto C_*(X;R)$ that we defined above also preserves  colimits, so it suffices to define it on representables.  The shifted chains functor, sending $X$ to $C_*(X;R)$ shifted so that $RX_{-1}$ is in degree 0, is characterized up to natural isomorphism by its value on the point $I(-,[0])$ together with the properties that it sends join of symmetric $\Delta$-complexes to tensor product of chain complexes, and preserves all colimits.
\end{remark}

\begin{remark}\label{rem:cones}
  Symmetric $\Delta$-complexes may themselves be regarded as special cases of the \emph{generalized cone complexes} of \cite[\S2]{acp}.  A cone is a topological space $\sigma$ together with an ``integral structure,'' i.e.,\ a finitely generated subgroup of the group of continuous functions $\sigma \to \R$ satisfying a certain condition, and a generalized cone complex is glued out of cones along certain maps: by definition they are presented as colimits of cones, indexed by any small category.
  
  The orthant $\R_{\geq 0}^{[p]}$ equipped with the group generated by the $p+1$ coordinate projections $\R_{\geq 0}^{[p]} \to \R_{\geq 0} \subset \R$ is an example of a cone.  The generalized cone complex corresponding to a symmetric $\Delta$-complex $X$ is then the colimit of orthants, indexed by $([p],x) \in J_X$.
\end{remark}

\subsection{The tropical moduli space as a symmetric $\Delta$-complex}
\label{sec:trop-moduli-space}

Let us return to the tropical moduli space $\Dg$, which we defined in
Section~\ref{sec:graphs}.  To illustrate how the definitions of this section
work for $\Dg$, we will give two descriptions that exhibit $\Dg$ as
the geometric realization of a symmetric $\Delta$-complex.  The
first description presents $\Dg$ as a colimit of a diagram of
symmetric $\Delta$-complexes; the second is an explicit description
as a functor $X\colon I^\mathrm{op} \to \Sets$.

The category $\Jg$ from \S\ref{sec:jgn} has a unique final object: a
single vertex, of weight $g$.  For the first description of $\Dg$ as a
colimit of a diagram of $\Delta$-complexes, choose for each object
$\Gmw \in \Jg$ a bijection $\tau = \tau_\Gmw\colon E(\Gmw) \to [p]$ for
the appropriate $p \geq -1$.  This chosen bijection will be called
the \emph{edge-labeling} of $\Gmw$.  The terminal object has $p = -1$,
and all non-terminal objects have $p \geq 0$.  We require no
compatibility between the edge-labelings for different $\Gmw$, but a
morphism $\phi\colon \Gmw \to \Gmw'$ determines an injection
\begin{equation*}
  [p'] \xrightarrow{\tau_{\Gmw'}^{-1}} E(\Gmw') \xrightarrow{\phi^{-1}} E(\Gmw)
  \xrightarrow{\tau_\Gmw} [p],
\end{equation*}
where the middle arrow is the induced bijection from the edges of
$\Gmw'$ to the non-collapsed edges of $\Gmw$ as in
Definition~\ref{def:E-and-H-as-functors}.  This gives a functor
$F\colon \Jg^\mathrm{op}\rightarrow I$ sending $\Gmw$ to the
codomain $[p]$ of $\tau_\Gmw$, and hence induces a functor from
$\Jg^\mathrm{op}$ to symmetric $\Delta$-complexes, given as
$\Gmw \mapsto I(-,F(\Gmw))$, whose colimit $X$ has geometric
realization $|X| = \Dg$ and cone $CX = M_g^\trop$.

Indeed, colimit commutes with geometric realization by Lemma~\ref{lem:geometric-realization-colimit}, so the geometric
realization of $X$ is the colimit of the functor
$\Gmw \mapsto |I(-,F(\Gmw))|$ from $\Jg^\mathrm{op}$ to $\mathsf{Top}$.  But
we have a homeomorphism $|I(-,F(\Gmw))| = \Delta^p$, where
$[p] = F(\Gmw)$.  Furthermore, if $\Gmw'\in \Jg$ is an object with
$p'+1$ edges, then the injection of label sets $F(\phi)\colon [p'] \to [p]$
determined as above by a morphism $\phi\colon \Gmw\to\Gmw'$, induces a
gluing of the simplex $|I(-,F(\Gmw'))| = \Delta^{p'}$ to a face of
$|I(-,F(\Gmw))| = \Delta^p$.  This agrees with the gluing obtained
from the gluing of $\sigma(\Gmw')$ to a face of $\sigma(\Gmw)$ in
Definition~\ref{definition:mgn} by restricting to the length-one
subspaces $\Delta^p \subset \sigma(\Gmw) = \R_{\geq 0}^{E(\Gmw)}$ and
$\Delta^{p'} \subset \sigma(\Gmw) = \R_{\geq 0}^{E(\Gmw')}$.

For the second description of $\Delta_{g}$ as the geometric
realization of a symmetric $\Delta$-complex, we explicitly describe
a functor $X\colon I^\mathrm{op} \to \Sets$ as follows. The
elements of $X_p$ are equivalence classes of pairs $(\Gmw,\tau)$
where $\Gmw\in \Jg$ and $\tau\colon E(\Gmw) \to [p]$ is an edge labeling;
two edge-labelings are considered equivalent if they are related by an
isomorphism $\Gmw \cong \Gmw'$ (including of course automorphisms).
Here $\Gmw$ ranges over all objects in $\Jg$ with exactly $p+1$
edges.  (Using that in \S\ref{sec:jgn} we tacitly picked one
element in each isomorphism class in $\Jg$, the equivalence relation
is generated by actions of the groups $\mathrm{Aut}(\Gmw)$.)  This
defines $X \colon I^\mathrm{op} \to \Sets$ on objects.

Next, for each injective map $\iota\colon [p']\to[p]$, define the
following map $X(\iota) \colon X_p \to X_{p'}$; given an element
of $X_p$ represented by $(\Gmw,\tau\colon E(\Gmw)\to [p])$,
contract the edges of $\Gmw$ whose labels are not in
$\iota([p'])\subset [p]$, then relabel the remaining edges with labels
$[p']$ as prescribed by the map $\iota$.  The result is a
$[p']$-edge-labeling of some new object $\Gmw'$, and we set
$X(\iota)(\Gmw)$ to be the element of $X_{p'}$ corresponding to it.

Hereafter, we will use $\Delta_g$ to refer to this symmetric $\Delta$-complex and write $|\Delta_g|$ for the topological space.  To avoid double subscripts, we write $\Delta_g([p])$ for the set of $p$-simplices of $\Delta_g$.   Then $H_k(\Delta_g) = \widetilde H_k(|\Delta_g|)$ since $\Delta_g([-1]) = \{\ast\}$.


\section{Graph complexes and cellular chains on $\Dg$}

In this section we prove Theorem~\ref{thm:gc}.
From \S\ref{sec:trop-moduli-space} we read off the following presentation of the cellular chain complex $C^{(g)} = C_*(\Delta_g;\Q)$.  There is one generator $[\Gmw,\omega]$ of degree $p$ for each object $\Gmw \in \Jg$ and each bijection $\omega$ from $E(\Gmw)$ to an object $[p]$ in $I$, subject to the relations $[\Gmw,\omega] = \mathrm{sgn}(\sigma)[\Gmw',\omega']$ if there exists an isomorphism $\Gmw \to \Gmw'$ in $\Jg$ inducing the permutation $\sigma$ of the set $[p] = \{0, \dots, p\}$.
  The differential is $$\partial[\Gmw,\omega] = \sum_{i=0}^p (-1)^i[\Gmw/e_i,\omega|_{\Gmw/e_i}],$$
where $\omega(e_i) = i$ and $\omega|_{\Gmw/e_i} \col E(\Gmw/e_i) \to [p-1]$ makes the following diagram commute.
 $$\xymatrix@R=6mm@C=10mm{E(\Gmw/e_i) \ar[r]^{\omega|_{\Gmw\!/\!e_i}}\ar[d]&[p\!-\!1]\ar[d]^{\delta^i}  \\ E(\Gmw) \ar[r]^{\omega} & [p]}
$$ 
Proposition~\ref{prop:cellularhomology} specializes to the following.
\begin{lemma}
  There is an isomorphism $\widetilde H_k(|\Delta_g|;\Q) \cong H_k(C^{(g)})$. \qed
\end{lemma}

\begin{definition}
  Let $B^{(g)} \subset C^{(g)}$ be the subcomplex spanned by those $[\Gmw,\omega]$ which have a vertex with positive weight, and let $A^{(g)} = C^{(g)} / B^{(g)}$ be the quotient chain complex.
\end{definition}

Hence $A^{(g)}$ is spanned by those $[\Gmw,\omega]$ in which $w(v)=0$ for all vertices $v$.  The boundary map on $A^{(g)}$ is still an alternating sum of edge contractions, now identifying $[\Gmw/e,\omega\vert_{\Gmw/e}] = 0$ when $e \in E(\Gmw)$ is a loop.  Comparing definitions, we get the following.

\begin{lemma}
  The chain complex $A^{(g)}$ is isomorphic to a shift of Kontsevich's graph complex $G^{(g)}$.  In our grading conventions, the isomorphism is
  \begin{align*}
    G^{(g)}_k & \to A^{(g)}_{k + 2g-1}\\
    \pushQED{\qed} 
    [\Gmw,\omega] & \mapsto [\Gmw,\omega].
                    \qedhere
                    \popQED
  \end{align*}
\end{lemma}

The next proposition implies that the homomorphism
\begin{equation*}
  \widetilde H_{k + 2g-1}(|\Delta_g|;\Q) \to H_k(G^{(g)})
\end{equation*}
induced by the quotient $C^{(g)} \twoheadrightarrow A^{(g)}$ is an isomorphism.
It is quite similar to the acyclicity result established in \cite[Theorem 2.2]{ConantGerlitsVogtmann05}.
\begin{proposition}
The chain complex $B^{(g)}$ has vanishing homology in all degrees.
\end{proposition}
  
  The complex $B^{(g)}$ calculates the reduced homology of the subspace of $\Delta_g$ consisting of graphs containing a vertex of positive weight.  In \cite{cgp3} we shall show that this space is in fact contractible.

\begin{proof} For any $\Gmw$ and any $e\in E(\Gmw)$, say $e$ is a {\em stem} if it separates a single weight-one vertex from the rest of $G$, i.e., if $G$ looks like
  \begin{equation*}
    {\scriptstyle 1}\bullet\! \stackrel{e}{\text{---}} \;?
  \end{equation*}
In particular a loop cannot be a stem.  Notice that for distinct edges $e,f\in E(\Gmw)$ such that $f$ is not a loop, $e$ is a stem in $\Gmw$ if and only if its image in $\Gmw/f$ is a stem.  Notice also that $\Gmw$ has a vertex of positive weight if and only if it admits a morphism from some $\Gmw'$ having a stem.  

Suppose $\Gmw$ has a vertex of positive weight. Let $\widetilde\Gmw$ be the graph obtained from $\Gmw$ as follows: for any vertex $v$ in $\Gmw$ of weight $w>0$, other than a vertex of weight $1$ already separated from the rest of $\Gmw$ by a stem, replace $v$ with a single vertex of weight zero incident to $w$ stems.  (If, and only if, $\Gmw$ is a single vertex of weight $2$, it is necessary to suppress the resulting unstable vertex of weight $0$ and valence 2 in $\widetilde\Gmw$.)  
There is a canonical morphism $\phi\col \widetilde{\Gmw}\to\Gmw$ which contracts all stems. Moreover, $\widetilde{\Gmw}$ is the unique ``maximal stem-uncontraction'' of $\Gmw$ in the following precise sense: for any morphism $\phi'\colon \Gmw' \to \Gmw$ that may be factored as a sequence of stem-contractions and  isomorphisms, there exists a morphism $\psi\colon \widetilde\Gmw \to \Gmw'$ with $\phi = \phi'\circ \psi$.

For $i\ge 0$, let $B^{(g),i}$ denote the subcomplex of $B^{(g)}$ spanned by graphs ${\Gmw}$ with at most $i$ edges that are {\em not} stems.  Then the subcomplexes $B^{(g),i}$, for $i=0,\ldots,3g-3$, filter $B^{(g)}$. 

Next, for each $i>0$, we claim vanishing of relative homology of the pair $(B^{(g),i}, B^{(g),i-1}).$  The chain complex associated to this pair is generated by those $[\Gmw,\omega]$ having a  positive weight, satisfying in addition that $\Gmw$ has exactly $i$ non-stem edges.  Furthermore, the boundary of $[\Gmw,\omega]$ is  a signed sum of 1-edge-contractions by stems.  
We claim that this chain complex is a direct sum of subcomplexes $B^{(g),i}(\Gmw),$ one for each $\Gmw$ with $i$ non-stems that is maximal in the sense discussed above.  Here, $B^{(g),i}(\Gmw)$ is the subcomplex of $B^{(g),i}$ with a generator in degree $p$ for each $[{\bf H},\omega]$ for $\bf H$ isomorphic to a stem-contraction of $\Gmw$ and $\omega\col E({\bf H}) \to [p]$ a bijection, with relations $[{\bf H},\omega] = \operatorname{sgn}(\sigma)[{\bf H}',\omega']$ whenever there is a $\sigma\in S_{p+1}$ and an isomorphism ${\bf H} \to {\bf H}'$ taking $\omega$ to $\sigma\circ\omega'.$
The fact that the chain complex does indeed split as a direct sum follows from the fact that if $\Gmw$ and $\Gmw'$ are both maximal stem-uncontractions of a given graph with a vertex of positive weight, then $\Gmw$ and $\Gmw'$ admit morphisms to each other, hence are isomorphic in $\Jg$.

Now we claim that each $B^{(g),i}(\Gmw)$ is acyclic. Indeed, we will show that it is isomorphic to the rational cellular chain complex associated to the pair 
$$(\Delta^{|E(\Gmw)|-1}/\!\Aut(\Gmw), Z/\!\Aut(\Gmw))$$
where $Z$ is the union of the $i$ facets of $\Delta^{|E(\Gmw)|-1}$ that contain all vertices of $\Delta^{|E(\Gmw)|-1}$ corresponding to stems of $\Gmw$. Since $0<i<|E(\Gmw)|,$  there is a natural deformation retraction of $\Delta^{|E(\Gmw)|-1}$ onto $Z$, and this retraction  is $\Aut(\Gmw)$-equivariant.  Therefore the cellular chain complex of the pair must be acyclic.  

The cellular chain complex of the pair $(\Delta^{|E(\Gmw)|-1}/\!\Aut(\Gmw), Z/\!\Aut(\Gmw))$ has a generator in degree $p$ for each $\Aut(\Gmw)$-equivalence class $[S,\omega]$, where $S\subseteq E(\Gmw)$ is a set of stem edges and $\omega\col E(\Gmw)\setminus S \to [p]$ is a bijection, with relations
$[S,\omega] = \operatorname{sgn}(\sigma)[S,\sigma\circ \omega]$ for $\sigma\in S_{p+1}$.   There is a map from this chain complex to $B^{(g),i}(\Gmw)$, sending $[S,\omega]\mapsto [\Gmw/S,\omega]$, and this map is evidently surjective. In fact it is an isomorphism: if $S,S'\subseteq E(\Gmw)$ are two sets of stem edges with $\Gmw/S\cong \Gmw/S'$, then by the maximality property of $\Gmw$ discussed above, there is a morphism $\Gmw\to \Gmw$ making the diagram
$$\xymatrix{\Gmw\ar[r]^\cong \ar[d] & \Gmw \ar[d] \\ \Gmw/S \ar[r]^\cong & \Gmw/S'}$$
commute. Moreover, any morphism $\Gmw\to \Gmw$ is an automorphism. Therefore $S$ and $S'$ are in the same $\Aut(\Gmw)$-orbit.  
This proves that $B^{(g),i}(\Gmw)$ is acyclic, as claimed.

It remains to see that $B^{(g),0}$ is acyclic.  Generators are graphs in which every edge is a stem. It is not hard to classify such $\Gmw$: if $g>2$, there is one isomorphism class for each $h \in \{0, \dots, g\}$, given by a graph with a single central vertex of weight $g-h$, to which are attached $h$ stems. (If $g=2$, then the classification is the same, but $h$ must be in $\{0,1\}$.)  For $h \geq 2$, these graphs all admit odd automorphisms, so the corresponding generator for the graph complex vanishes.  Thus $B^{(g),0}$ is rank two, generated by the $h =1$ and $h = 0$ graphs, and the boundary of the $h=1$ graph is the $h=0$ graph.
\end{proof}

\section{Boundary complexes} \label{sec:stackboundary}

The theory of dual complexes for simple normal crossings divisors is
well-known.  They may be constructed as $\Delta$-complexes, with the $\Delta$-complex structure depending on a choice of total ordering on the irreducible components of the divisor.  Many applications
involve the fact that the homotopy types (and even simple homotopy types) of \emph{boundary complexes},
the dual complexes of boundary divisors in simple normal crossings
compactifications, are independent of the choice of compactification.  The same is also true for Deligne-Mumford (DM) stacks \cite{Harper17}.
Boundary complexes were introduced and studied by Danilov in the 1970s
\cite{Danilov75}, and have become an important focus of research
activity in the past few years, with new connections to Berkovich
spaces, singularity theory, geometric representation theory, and the
minimal model program.  See, for instance, \cite{Stepanov08, ABW13,
  boundarycx, KollarXu16, Simpson16, deFernexKollarXu17}.

In order to apply combinatorial topological properties of $\Dg$ to
study the moduli space of curves $\cM_{g}$ using the
compactification by stable curves, we must account for the facts that
$\cM_{g}$ and $\ocM_{g}$ are stacks, not varieties, and that the
boundary divisor in $\ocM_{g}$ has normal crossings, but not simple
normal crossings.  The latter of those two complications is the more
serious one; when the irreducible components of the strata have
self-intersections, the fundamental groups of strata may act
nontrivially by monodromy on the analytic branches of the boundary and
this needs to be accounted for. Once that is
properly understood, passing from varieties to stacks is relatively
straightforward.

In this section we explain how dual complexes of normal crossings
divisors are naturally interpreted as symmetric $\Delta$-complexes
and, in particular, the dual complex of the boundary divisor in the
stable curves compactification of $\cM_{g}$ is naturally identified
with~$\Dg$.  

\subsection{Dual complexes of simple normal crossings
  divisors} \label{sec:dual-compl-simple}

We begin by recalling the notion of dual complexes of simple normal
crossings divisors, using the language of symmetric $\Delta$-complexes introduced in Section~\ref{sec:cellular}.  In
\S\ref{sec:dual-compl-norm}, we will explain how to interpret dual
complexes of normal crossings divisors in smooth DM
stacks as symmetric $\Delta$-complexes of \S\ref{sec:cellular-chains}. Here and throughout,
all of the varieties and stacks that we consider are over the complex
numbers, and all stacks are separated and DM.

Let $X$ be a $d$-dimensional smooth variety.  Recall (cf.\ \cite[Tag 0BI9]{stacks-project}) that a
\emph{simple normal crossings} divisor is an effective Cartier divisor  $D \subset X$ which is Zariski locally cut out by $x_1 \cdots x_d$ for a regular system of parameters $x_1, \dots, x_d$ in the local ring at any $p \in D$.  
The {\em strata} of $D$ may be defined inductively as
follows.  The $(d-1)$-dimensional strata of
 $D$ are the irreducible components of the smooth locus of $D$. For
each $i<d-1$, the $i$-dimensional strata are the irreducible
components of the regular locus of the complement of the union of all strata of $D$ of dimension greater than $i$.

If $D \subset X$ has simple normal crossings, then the dual complex
$\Delta(D)$ is naturally understood as a regular symmetric
$\Delta$-complex whose geometric realization has one vertex for each
irreducible component of $D$, one edge for each irreducible component
of a pairwise intersection, and so on. The inclusions of faces
correspond to containments of strata.  It is augmented, with
$(-1)$-simplices the set of irreducible components (equivalently,
connected components) of $X$.  Equivalently, using our
characterization of symmetric $\Delta$-complexes in terms of
presheaves on the category $I$ given in \S\ref{sec:symmetric},
$\Delta(D)$ is the presheaf whose value on $[p]$ is the set of pairs
$(Y, \phi)$, where $Y \subset X$ is a stratum of codimension $p+1$,
i.e., codimension $p$ in $D$ for $p \geq 0$, and $\phi$ is an ordering
of the components of $D$ that contain $Y$, with maps induced by
containments of strata. Dual complexes can also be defined in exactly
the same way for simple normal crossings divisors in DM stacks.

\begin{remark}
  In the literature, it is common to fix an ordering of the
  irreducible components of the simple normal crossings divisor $D$.
  The corresponding ordering of the vertices induces a
  $\Delta$-complex structure on $\Delta(D)$.  Working with dual
  complexes as symmetric $\Delta$-complexes may be slightly more natural, in that
  it avoids this choice of an ordering, and certainly it generalizes better to the
  construction of dual complexes for divisors with (not necessarily simple) normal crossings as symmetric
  $\Delta$-complexes, given in \S\ref{sec:dual-compl-norm}.

  In the literature it is also commonly assumed that $X$ is
  irreducible, and hence there is no need for keeping track of
  $(-1)$-simplices and augmentations.  This is sufficient for studying
  one irreducible variety at a time, but comes with some
  technical inconveniences.  In particular, certain auxiliary
  constructions, such as the \'etale covers and fiber products appearing later in this
  section, do not preserve irreducibility.  It is convenient to set up
  the language in a way that applies without assuming
  irreducibility.
\end{remark}

\subsection{Dual complexes of normal crossings divisors}
\label{sec:dual-compl-norm}

We now discuss the generalization to normal crossings divisors $D$ in
a smooth DM stack $X$ which are not necessarily simple normal
crossings, i.e.,\ the irreducible components of $D$ are not
necessarily smooth and may have self-intersections.  This situation is
more subtle, even for varieties, due to monodromy.  In the stack
case, when the boundary strata have stabilizers, the monodromy
may be nontrivial even for zero-dimensional strata. This phenomenon
appears already at the zero-dimensional strata of
$\overline {\mathcal{M}}_{g}$ given by stable curves having nontrivial
automorphisms, i.e.,\ the strata corresponding to (unweighted)
trivalent graphs of first Betti number $g$ with nontrivial
automorphisms.

Let $X$ be a smooth variety or DM stack, not necessarily irreducible.
Recall that a divisor $D \subset X$ has normal crossings if and only
if there is an \'etale cover by a smooth variety $X_0 \rightarrow X$
in which the preimage of $D$ is a divisor with simple normal
crossings.  
In this situation, $D$ is a simple normal crossings divisor if all irreducible components are smooth.
Note that this \'etale local characterization of normal
crossings divisors is the same for varieties and DM stacks.

In this situation the dual complex may be defined directly as a functor
$I^\mathrm{op} \to \Sets$, in the following way.  Let $\widetilde{D} \to X$ denote the normalization of $D \subset X$, and for $[p] \in I$ write
\begin{equation*}
  \widetilde{D}_p = (\widetilde{D} \times_X \dots \times_X \widetilde{D}) \setminus \{(z_0, \dots, z_p) \mid \text{$z_i = z_j$ for some $i \neq j$}\}.
\end{equation*}
We have $\widetilde{D}_0 = \widetilde{D}$ and $\widetilde{D}_{-1} = X$.  Then $\widetilde{D}_p \to X$ is a local complete intersection morphism whose conormal sheaf is a vector bundle of rank $(p+1)$ (\cite[Tag 0CBR]{stacks-project}).  In particular $\widetilde{D}_p$ is smooth over $\C$ of dimension $d-p$ if $X$ is smooth over $\C$ of dimension $d+1$.

\begin{definition}\label{defn:boundary-complex-as-generalized-Delta}
  Let $X$ be a smooth variety or DM stack, let $D \subset X$ be a
  normal crossings divisor, and write $\widetilde{D}_p \to X$ for the
  construction defined for all $[p] \in I^\mathrm{op}$ above.  In this situation, define
  the symmetric $\Delta$-complex $\Delta(D)$ by letting $\Delta(D)_p$
  be the set of irreducible components (= connected components) of $\widetilde{D}_p$.
\end{definition}  
\noindent We point out that in the case of stacks, the association $[p] \mapsto \widetilde{D}_p$ will only be a pseudofunctor, but the set of irreducible components will be functorial in $[p] \in I$.

We note that a closed point of $\widetilde{D}_p$ corresponds precisely to a closed point $x$ in a codimension $p$ stratum of $D$, together with an ordering $\sigma$ of the $p+1$ local analytic branches of $D$.  Hence $\Delta(D)_p$ may be described more transcendentally as the set of equivalence classes of pairs $(x, \sigma)$, where $(x,\sigma)$ is equivalent to $(x',\sigma')$ if there is a path (continuous in the analytic topology) within the stratum connecting $x$ to $x'$, and that following the ordering of the branches along this path takes $\sigma$ to $\sigma'$.

\begin{remark}\label{remark:simplicial-complexes}
  Recall that a $\Delta$-complex $X$ is \emph{regular} if the maps
  $\Delta^p \to |X|$ associated to $\sigma \in X_p$ for all
  $p \geq 0$ are all injective.  This definition makes sense equally
  well for symmetric $\Delta$-complexes $X$ and is equivalent to the
  condition that every edge of $X$ has two distinct endpoints, i.e.,
  for any $e \in X_1$, we have $d_0(e) \neq d_1(e)$.

  The dual complex of a normal crossings divisor will be a regular symmetric
  $\Delta$-complex exactly when $D$ has \emph{simple} normal
  crossings, meaning that every irreducible component of $D$ is
  smooth. 
  Indeed, the irreducible components of $D$ are smooth if and only if at every codimension $1$ stratum of $D$, the two analytic branches belong to distinct irreducible components.  This is equivalent to the condition that $d_0(e)\ne d_1(e)$ for $e\in \Delta(D)_1$. 
    \end{remark}

The comparison of Definition~\ref{defn:boundary-complex-as-generalized-Delta} with \cite{acp} uses the following \'etale descent result, whose proof we omit (see \cite{CGP1v1}).  Let $D$ and $D'$ be normal crossings divisors in $X$ and $X'$, respectively.  Suppose $\pi \colon X' \to X$ is \'etale and $D'$ is the preimage of $D$. Then $\pi$ induces natural maps $\widetilde D'_p \to \widetilde D_p$ for all $p$, and hence a morphism $\Delta(D') \to \Delta(D)$. Thus $(X,D) \mapsto \Delta(D)$ is functorial for such \'etale maps.  Together with the discussion in \S\ref{sec:dual-compl-simple}, the following \'etale descent property completely characterizes this functor.

\begin{lemma}\label{lemma:etale-descent} With boundary complexes defined as in Definition~\ref{defn:boundary-complex-as-generalized-Delta},
  the association $(X,D) \mapsto \Delta(D)$ satisfies \'etale descent in
  the sense that if $X_0 \to X$ is an \'etale cover and
  $X_1 =X_0 \times_X X_0$, then
  \begin{equation*}
    \Delta(D\times_X X_1) \double \Delta(D \times_X X_0) \to\Delta(D)
  \end{equation*}
  is a coequalizer diagram.
\end{lemma}

\begin{example}\label{ex:whitney} Consider the {\em Whitney umbrella}
  $D = \{x^2y=z^2\}$ in $X=\mathbb{A}^3\setminus \{y=0\},$ as in
  \cite[Example 6.1.7]{acp}.  Calculating the dual complex $\Delta(D)$ by descent along the degree 2 \'etale cover given by a base change $y = u^2$ gives the
  half interval of Example~\ref{ex:halfinterval}, presented as a quotient of $I(-,[1])$ by the action of $\Z/2\Z$.
\end{example}

\begin{corollary} \label{cor:Delta-to-cone}
  Let $X$ be a smooth variety or DM stack with the toroidal structure
  induced by a normal crossings divisor $D \subset X$.  Then the dual
  complex $\Delta(D)$ is the symmetric $\Delta$-complex associated
  to the smooth generalized cone complex $\Sigma(X)$.
\end{corollary}
\begin{proof}
Let $X_0 \to X$ be an \'etale cover such that $D_0 = D \times_X X_0$ has simple normal crossings in $X_0$.  Let $X_1 = X_0 \times_X X_0$ and $D_1 = D \times_X X_0$. Endow $X_i$ with the toroidal structure induced by the simple normal crossings divisor $D_i$. Then $\Sigma(X_i)$ is the cone over $\Delta(D_i)$, and \cite[Proposition~6.1.2]{acp} describes $\Sigma(X)$ as the coequalizer in generalized cone complexes of $\Sigma(X_1) \double \Sigma(X_0)$.  The result then follows from Lemma~\ref{lemma:etale-descent}, since the identification of symmetric $\Delta$-complexes with smooth
 generalized cone complexes preserves all colimits.
\end{proof}

Most important for our purposes is the special case where
$X = \ocM_{g}$ is the Deligne--Mumford stable curves
compactification of $\cM_{g}$ and
$D = \ocM_{g} \smallsetminus \cM_{g}$ is the boundary divisor.   Modulo the translation from symmetric $\Delta$-complexes to (smooth) generalized cone complexes given in Corollary \ref{cor:Delta-to-cone}, the following statement is one of the main results of \cite{acp} and we refer there for details.  

\begin{corollary}\label{cor_dual_complex}
  The dual complex of the boundary divisor in the moduli space of
  stable curves with marked points
  $\Delta(\ocM_{g} \smallsetminus \cM_{g})$ is $\Dg$.
\end{corollary}

\subsection{Top weight cohomology}   \label{sec:topweight}

Let $\cX$ be a smooth variety or DM stack of dimension $d$ over $\C$.  The rational singular cohomology of $\cX$, like the rational cohomology of a smooth variety, carries a canonical mixed Hodge structure, in which the weights on $H^k$ are between $k$ and $\min\{2k,2d\}$.  Since the graded pieces $\Gr_j^W H^*(\cX; \Q)$ vanish for $j > 2d$, we refer to $\Gr_{2d}^W H^*(\cX; \Q)$ as the \emph{top weight cohomology} of $\cX$. 

\begin{theorem} \label{thm:topweight}
Let $\cX$ be a smooth and separated DM stack of dimension $d$ with a normal crossing compactification $\ocX$ and let $\cD = \ocX \smallsetminus \cX$.  Then there is a natural isomorphism
\[
\Gr_{2d}^W H^{2d-k}(\cX; \Q) \  \cong  \  H_{k-1}(\Delta (\cD); \Q),
\]
whose codomain is $\widetilde H_{k-1}(|\Delta(\cD)|;\Q)$ when $X$ is irreducible.
\end{theorem}

\begin{proof}
First, we reduce to the case where $\cD$ has simple normal crossings, by a finite sequence of blowups, as follows.  Let $\ocX' \rightarrow \ocX$ be the morphism obtained by first blowing up the zero-dimensional strata of $\cD$, and then the strict transforms of the 1-dimensional strata, and so on.  Then one readily checks that $\cD' = \ocX' \smallsetminus \ocX$ has simple normal crossings and $\Delta(\cD')$ is the barycentric subdivision of $\Delta(\cD)$, as defined in \S\ref{sec:subdivision}.

We may therefore assume that $\cD$ has simple normal crossings.   The remainder of the argument is essentially identical to
the proof for simple normal crossings divisors in algebraic varieties given in \cite[Sections~2 and 4]{boundarycx}. The one additional fact needed to go from varieties to DM stacks is that the cohomology of a smooth DM stack $\ocX$ with projective coarse moduli space $X$ is pure. To see this, note that the natural map $\ocX \to X$ induces an isomorphism $H^*(X;\Q) \to H^*(\ocX;\Q)$ (see \cite{Behrend04} or \cite[Theorem~4.40]{Edidin13}) and, since $X$ is a compact K\"ahler $V$-manifold, its cohomology is pure \cite[Theorem~2.43]{PetersSteenbrink08}.

We briefly recall the argument: there is a complex of $\Q$-vector spaces
\[
0 \rightarrow \bigoplus_{i=1}^r H^j(\cD_i;\Q) \xrightarrow{\delta_0} \bigoplus_{i_0 < i_1} H^j(\cD_{i_0} \times_{\ocX}  \cD_{i_1};\Q) \xrightarrow{\delta_1} \bigoplus_{i_0 < i_1 < i_2} H^j(\cD_{i_0}  \times_{\ocX} \cD_{i_1} \times_{\ocX} \cD_{i_2};\Q) \xrightarrow{\delta_2} \cdots,
\]
whose cohomology gives the $j$-graded pieces of the weight filtrations on the cohomology groups of $\cD$; see \cite[Chapter~4, \S2]{KulikovKurchanov98}.  Set $j=0$; then the long exact sequence associated to the pair $(\ocX,\cD)$, together with the fact that $H^k(\ocX,\Q)$ is pure of weight $k$, allows one to relate the cohomology of the above complex to $\Gr^W_0H^\bullet_c(\cX;\Q)$, which is isomorphic by Poincar\'e duality to $\Gr^W_{2d}H^{2d-\bullet}(\cX;\Q)^\vee$.  
\end{proof}

\section{Applications} \label{sec:applications}

We now proceed to use the identification of top weight cohomology of $\cM_g$ with reduced homology of the symmetric $\Delta$-complex $\Delta_g$ developed in the preceding sections, 
 in combination with known nonvanishing and vanishing results for graph homology and cohomology of $\cM_g$, to prove the applications stated in the introduction.

{ \renewcommand{\thetheorem}{\ref{thm:comparison}}
\begin{theorem}  
There is an isomorphism
\[
\Gr_{6g-6}^W H^{6g-6-k} (\cM_{g}; \Q) \xrightarrow{\cong} \widetilde{H}_{k-1}(|\Dg|;\Q) ,
\]
identifying the reduced rational homology of $\Dg$ with the top graded piece of the weight filtration on the cohomology of $\cM_{g}$.
\end{theorem}  
\addtocounter{theorem}{-1} }

\begin{proof}
Let $\cD = \ocM_g \smallsetminus \cM_g$.  Then $\Delta_g$ is naturally identified with the dual complex $\Delta(\cD)$, by Corollary~\ref{cor_dual_complex}.  The theorem is therefore the special case of Theorem~\ref{thm:topweight} where $\cX = \cM_g$ and $\ocX = \ocM_g$.
\end{proof}

We now prove our nonvanishing result for $H^{4g-6}(\cM_g;\Q)$.

{ \renewcommand{\thetheorem}{\ref{thm:nonvanishing}}
\begin{theorem}
The cohomology $H^{4g-6}(\cM_g;\Q)$ is nonzero for $g =3$, $g = 5$, and $g \geq 7$.  In fact, $\dim H^{4g-6}(\cM_g;\Q)$ grows at least exponentially; precisely, $$\dim H^{4g-6}(\cM_g;\Q) > \beta^g + \text{constant}$$ for any 
$\beta< \beta_0$, where $\beta_0 \approx 1.3247\ldots$ is the real root of $t^3-t-1=0$. 

\end{theorem}
\addtocounter{theorem}{-1} }

\begin{proof}
By Theorems~\ref{thm:comparison} and~\ref{thm:gc}, we have a natural surjection $H^{4g-6}(\cM_g) \rightarrow H_0(G^{(g)})$.  
Therefore the result follows from Theorem~\ref{thm:nonvanishing-graph-homology}.
\end{proof}

Note that the nonvanishing unstable cohomology group $\Gr_{12}^W H^6(\cM_3;\Q)$ found by Looijenga \cite{Looijenga93} is identified with the span of $[W_3]$ in $H_0(G^{(3)})$.  Hence, the nonvanishing, unstable, top weight cohomology that we describe, especially those corresponding to the spans of $[W_g]$ for odd $g \geq 5$, may be seen as direct and natural generalizations.

In comparison, the asymptotic size of the tautological ring of $\mathcal{M}_g$ is bounded above by $C^{\sqrt{g}}$ for a constant $C$.  Indeed, its Poincar\'e series is dominated coefficient-wise by that of the polynomial ring
$$\Q[\kappa_1,\kappa_2,\ldots],\qquad \operatorname{deg} \kappa_i = 2i.$$
where $\kappa_i$ has degree $2i$, and $\mathcal{M}_g$ has virtual cohomological dimension $4g-5$ \cite{Harer86}.  A rough bound may be obtained by calculating $\dim \Q[\kappa_1,\kappa_2,\ldots]_{2n} = p(n)$ where $p(n)$ is the number of partitions of $n$.  Since $\log(p(n)) \sim \pi \sqrt{2n/3}$ by \cite[eq.~(5.22)]{HR1917}, the dimension of the tautological ring is bounded by $\sum_{n=1}^{2g-3}p(2n) < 2g\cdot p(2g) < A\cdot B^{\sqrt{g}} + C$ for constants $A$, $C$ and $C$.

On the other hand, the Euler characteristic estimates by Harer--Zagier mentioned earlier imply that the size of the top weight part of $H^{4g-6}(\cM_g;\Q)$ as $g \to \infty$ accounted for here is still negligible in comparison to the entire $H^*(\cM_g;\Q)$ (and hence in comparison to the largest single Hodge number of $\cM_g$).

We also record the following nonvanishing result of odd-degree cohomology groups, as discussed in the introduction:

\begin{corollary}
The cohomology groups $H^{15}(\cM_6;\Q)$, $H^{23}(\cM_8;\Q)$, and $H^{27}(\cM_{10};\Q)$ are nonzero.
\end{corollary}
\begin{proof}
By Theorems~\ref{thm:comparison} and~\ref{thm:gc}, the nontrivial classes in $H_3(G^{(6)})$, $H_3(G^{(8)})$, and $H_7(G^{(10)})$ discovered computationally in \cite{BarNatanMcKay} implies nonvanishing of $H^{15}(\cM_6;\Q)$, $H^{23}(\cM_8;\Q)$, and $H^{27}(\cM_{10};\Q)$. 
\end{proof}
The computations of \cite{BarNatanMcKay} are extended, albeit by approximate (floating point) calculations, in \cite{KhoroshkinWillwacherZivkovic17}, where it is also shown that $\bigoplus_{j\ge 0} H^{2j+1}(\mathsf{GC})$ is infinite dimensional.    

Combining \cite[Corollary 6]{KhoroshkinWillwacherZivkovic17} with Theorems~\ref{thm:comparison}, \ref{thm:gc}, and \ref{thm:nonvanishing-graph-homology} yields the following dimension bound on top-weight odd-degree cohomology.

\begin{corollary} For each $g\ge2$, we have:
\begin{equation*}
\sum_{\substack{g\ge g' \ge (2g+2)/3,\,\,i\ge 0}} 
\dim\Gr_{6g'-6}^W H^{2i+1} (\cM_{g'}; \Q) > \beta^g  + \text{constant},
\end{equation*}
for any
$\beta< \beta_0$, where $\beta_0 \approx 1.3247\ldots$ is the real root of $t^3-t-1=0$.  
\end{corollary}
\medskip

We conclude with an application in the other direction, using known vanishing results for $\cM_g$ to reprove a recent vanishing result of Willwacher for graph homology.

{ \renewcommand{\thetheorem}{\ref{thm:graphhom}}
\begin{theorem}
The graph homology groups $H_k(G^{(g)})$ vanish for $k < 0$.
\end{theorem}
\addtocounter{theorem}{-1} }

\begin{proof}
The virtual cohomological dimension of $\cM_g$ is $4g-5$ \cite{Harer86}.  Furthermore, $H^{4g-5}(\cM_g;\Q)$ vanishes \cite{ChurchFarbPutman12, MoritaSakasaiSuzuki13}.  Therefore $H^{4g-6-k}(\cM_g;\Q)$ vanishes for $k < 0$.  The theorem follows, since $H^{4g-6-k}(\cM_g;\Q)$ surjects onto  $H_{k}(G^{(g)})$.
\end{proof}

\section{Hyperbolic surfaces}

Here, we use the hyperbolic model for $\cM_g$ to construct the proper map $\ellmap\colon \cM_g \to M_g^\trop$ announced in~\eqref{eq:4}, and to give an explicit interpretation of the classes in $H_{4g-6}(\cM_g;\Q)$ arising from our main results.  More details will appear in a sequel.

In the hyperbolic interpretation of $\cM_g$, points are hyperbolic metrics on a closed genus $g$ surface $\Sigma$, up to isometric diffeomorphisms.  
Let $\cM_g^\mathrm{thick} \subset \cM_g$ denote the subspace given in the hyperbolic model for $\cM_g$ as those hyperbolic surfaces in which no non-trivial geodesic has length less than $\epsilon$, for a suitably small $\epsilon > 0$.  Then $\cM_g^{\mathrm{thick}} \subset \cM_g$ is a deformation retract \cite[p.~476]{HarerZagier86}.  Equivalently, Harvey's Borel--Serre type compactification of $\cM_g$ or the Kato--Nakayama space associated to the boundary divisor in $\ocM_g$ may be used instead of $\cM_g^\mathrm{thick}$.
Its boundary consists of hyperbolic surfaces with at least one geodesic of length $\epsilon $, but it is better regarded as an orbifold with corners: it is covered by  orbifold charts of the form
\begin{equation*}
  \R_{\geq 0}^S \times \R^T \to \cM_g^\mathrm{thick}
\end{equation*}
for finite sets $S$ and $T$ (varying from chart to chart).

\subsection{A map of spaces}
\label{sec:induced-by-map-of-spaces}

Let $h$ be a hyperbolic metric on a closed oriented 2-manifold $\Sigma$ of genus $g$ and let $\Gamma$ be the dual graph of the nodal 2-manifold obtained from $\Sigma$ by collapsing all closed geodesics of length smaller than some suitable $\epsilon$, chosen once and for all.  Each $e \in E(\Gamma)$ then corresponds to a simple closed geodesic in $(\Sigma,h)$ of length $a_e< \epsilon$, and we let $\ell(e) = -\log(a_e/\epsilon)$.  For sufficiently small $\epsilon > 0$, this recipe $[\Sigma,h] \mapsto (\Gamma,\ell)$ defines a (well defined) proper map
\begin{equation*}
  \lambda\colon \cM_g \to M_g^\trop.
\end{equation*}
This is a model for the map~\eqref{eq:4} in the introduction.

\subsection{Generalizations of abelian cycles}
\label{sec:gener-abel-cycl}

The injection $H_k(G^{(g)})^\vee \to H_{4g-6-k}(\cM_g;\Q)$ allows us to produce non-zero homology classes in the mapping class group from classes in $\mathfrak{grt}_1 \cong H^0(\mathsf{GC}) \cong \prod_g H_0(G^{(g)})^\vee$.  It is natural to ask for a more explicit description of the resulting homology classes.  In this section we shall outline how to transport a class represented by a cocycle $\alpha\colon G^{(g)}_k \to \Q$ through these isomorphisms.

\newcommand{\usedtobeA}{K}  

\begin{definition}
  For $p \geq -1$, let $\usedtobeA_p$ be the space of isometry classes of pairs consisting of a hyperbolic genus $g$ surface in $\cM_g^\mathrm{thick}$ together with an ordered $(p+1)$-tuple of distinct geodesics of length $\epsilon$, considered up to isometry of the surface preserving the ordered tuple of geodesics.  Then $S_{p+1} = I([p],[p])$ acts on $\usedtobeA_p$ by permuting the geodesics, and $d_i \colon \usedtobeA_{p+1} \rightarrow \usedtobeA_p$ is induced by forgetting the $i^{\mbox{th}}$ geodesic, defining a functor $I^\mathrm{op} \to \mathrm{Spaces}$.

  In particular, $\usedtobeA_{-1}$ is the coarse space of the orbifold $\cM_g^\mathrm{thick}$.  Let $\partial^p \cM_g^\mathrm{thick}$ denote the image of the map $\usedtobeA_p \to \usedtobeA_{-1}$ induced by $\emptyset \subset [p]$.
\end{definition}

The symmetric $\Delta$-complex defined as $[p] \mapsto \pi_0(\usedtobeA_p)$ is isomorphic to $\Delta_g$.  This may be seen by identifying the orbifold underlying $\usedtobeA_p$ with an $(S^1)^{p+1}$-bundle over the complex analytic orbifold underlying $\widetilde{D}_p \smallsetminus d_0(\widetilde{D}_{p+1})$, up to homotopy, or, more directly, by sending a hyperbolic surface with $(p+1)$ ordered labeled geodesics to the dual graph of the nodal 2-manifold obtained by collapsing the geodesics.

A cochain $G^{(g)}_p \to \Q$ is naturally identified (by extending to zero on graphs with non-zero weights) with a cochain $\alpha \in C^p(\Delta_g;\Q)$.  By definition, such a cochain is a function $\alpha\colon \Delta_g([p]) = \pi_0(\usedtobeA_p) \to \Q$ which is alternating under the action of $S_{p+1}$ on $\usedtobeA_p$.  Hence we may regard such a cochain as an element $\alpha \in H^0(\usedtobeA_p;\Q)$
on which a permutation $\sigma \in S_{p+1}$ acts as $\mathrm{sgn}(\sigma)$.  Such a cochain is a cocycle exactly when it is in the kernel of 
\numberwithin{equation}{section}
\begin{equation}\label{eq:new-7.1}
  (H^0(\usedtobeA_p;\Q) \otimes \Q^{\mathrm{sgn}})^{S_{p+1}} \xrightarrow{\sum (-1)^i (d_i)^*} (H^0(\usedtobeA_{p+1};\Q) \otimes \Q^{\mathrm{sgn}})^{S_{p+2}}.
\end{equation}

Next we wish to apply Poincare duality to $K_p$, which is a compact rational homology manifold with boundary.  We first pin down orientations, i.e., fundamental classes $[K_p] \in H_{d-p}(K_p, \partial K_p)$.  The subset $d_i(\partial K_{p+1}) \subset \partial K_p$ is independent of $i$, and we have homomorphisms
\begin{equation*}\label{eq:2}
    H_{d-p}(K_p,\partial K_p) \xrightarrow{\delta} H_{d-p-1}(\partial \usedtobeA_p,d_i(\partial \usedtobeA_{p+1})) \xleftarrow{(d_i)_*}
    H_{d-p-1}(\usedtobeA_{p+1},\partial \usedtobeA_{p+1}),
\end{equation*}
where we write $\delta$ for the connecting homomorphism of the triple.  The map $(d_i)_*$ is an isomorphism by excision, and the orientations are chosen such that $[K_{p+1}] = ((d_0)_*)^{-1} \circ \delta ([K_p])$, which forces $((d_i)_*)^{-1} \circ \delta ([K_p]) = (-1)^i [K_{p+1}]$ and $\sigma_*([K_p]) = \mathrm{sgn}(\sigma)[K_p]$ for $\sigma \in S_{p+1}$.

%
%
Poincar\'e duality, i.e., cap product with these fundamental classes, now identifies the homomorphism in~\eqref{eq:new-7.1} with a homomorphism
\numberwithin{equation}{section}
\begin{equation}\label{eq:16}
    H_{d-p}(\usedtobeA_p,\partial \usedtobeA_p;\Q)_{S_{p+1}} \xrightarrow{\sum ((d_i)_*)^{-1} \circ \delta} H_{d-p-1}(\usedtobeA_{p+1},\partial \usedtobeA_{p+1};\Q)_{S_{p+2}},
\end{equation}
where the signs in both the $S_{p+1}$ action and the boundary homomorphism have canceled with those in the fundamental classes. A cocycle $\alpha \in C^p(\Delta_g;\Q)$ gives a Poincar\'e dual $\mathrm{PD}([\alpha]) \in H_{d-p}(\usedtobeA_p,\partial \usedtobeA_p;\Q)_{S_{p+1}}$ in the kernel of~(\ref{eq:16}).  Mapping into $\usedtobeA_{-1}$ sends all spaces into $\partial^p \cM_g^\mathrm{thick}$ and the map~(\ref{eq:16}) fits into a commutative square
\begin{equation*}
  \xymatrix{
    H_{d-p}(\usedtobeA_p,\partial \usedtobeA_p;\Q)_{S_{p+1}} \ar[d]^{1/(p+1)!}\ar[r]^-{(\ref{eq:16})} & H_{d-p-1}(\usedtobeA_{p+1},\partial \usedtobeA_{p+1};\Q)_{S_{p+2}} \ar[d]^{1/(p+2)!}\\
    H_{d-p}(\partial^p \cM_g^\mathrm{thick},\partial^{p+1} \cM_g^\mathrm{thick};\Q) \ar[r] &
    H_{d-p-1}(\partial^{p+1} \cM_g^\mathrm{thick},\partial^{p+2} \cM_g^\mathrm{thick};\Q),
    }
\end{equation*}
where the bottom row is the connecting homomorphism for the triple.  The class $\mathrm{PD}([\alpha])$ in the upper left corner therefore maps to a class in $H_{d-p}(\partial^p \cM_g^\mathrm{thick},\partial^{p+1} \cM_g^\mathrm{thick})$, which admits a lift to homology relative to $\partial^{p+2} \cM_g^\mathrm{thick}$.  Since that space has no homology above degree $(d-p-2)$, another long exact sequence shows that this class lifts uniquely to $H_{d-p}(\partial^p \cM_g^\mathrm{thick};\Q)$.  By a similar argument, one checks that the image of this class in $H_{d-p}(\partial^{p-1} \cM_g^\mathrm{thick};\Q)$ is unchanged by adding a coboundary to $\alpha$, and hence one gets a well defined class in $H_{d-p}(\cM_g^\mathrm{thick};\Q)$ depending only on the cohomology class $[\alpha] \in H^p(\Delta_g;\Q)$.

In the special case where $p = 3g-4$, generators of $C^p(\Delta_g;\Q)$ are trivalent graphs and automatically cocycles since $C^{p+1}(\Delta_g;\Q) = 0$, and the resulting classes in $$H_{3g-3}(\cM_g;\Q) = H_{3g-3}(\Mod_g;\Q)$$ are exactly the abelian cycles associated to maximal collections of commuting Dehn twists.  In this way, homology classes on $\cM_g$ associated to graph cohomology classes in $H^*(G^{(g)})$ may be seen as generalizations of abelian cycle classes for the mapping class group.

\bibliographystyle{amsalpha}
\bibliography{math}

\end{document}